\documentclass{article}
\usepackage{graphicx} 

\usepackage[left=1in,top=1in,right=1in,bottom=1in,nohead]{geometry}
\usepackage{multirow}
\usepackage{url,subfig,latexsym,amsmath,amssymb, amsfonts,enumerate,
  epsfig, graphics,times, amsthm,mathtools, bbm}

\usepackage{tikz, pgfplots}
\usetikzlibrary{automata,positioning,arrows, intersections, calc}
\pgfplotsset{compat=1.8}

\usepackage{commath}
\usepackage{wrapfig}
\usepackage{xcolor}
\usepackage{nameref}
\usepackage{listings}
\usepackage[left]{showlabels}
\usepackage[colorlinks]{hyperref}
\usepackage{subcaption}
\usepackage{makeidx}
\usepackage{tabularx}
\usepackage{array}
    \newcolumntype{P}[1]{>{\centering\arraybackslash}p{#1}}
    \newcolumntype{M}[1]{>{\centering\arraybackslash}m{#1}}
\usepackage{comment}

\usetikzlibrary {graphs}
\usepackage{chemfig}
\newtheorem{thm}{Theorem}[section]
\newtheorem{cor}[thm]{Corollary}

\newtheorem{prop}[thm]{Proposition}

\newtheorem*{thm*}{Theorem}

\theoremstyle{definition}
\newtheorem{defn}{Definition}[section]
\newtheorem{rem}{Remark}[section]
\newtheorem{assumption}{Assumption}[section]

\def\S{\mathbb{S}}

\newcommand{\dlim}{\displaystyle \lim}
\newcommand{\dlimsup}{\displaystyle \limsup}
\newcommand{\dliminf}{\displaystyle \liminf}

\newcommand{\dmin}{\displaystyle \min}

\def\T1{T^1_{\{x_n\}}}
\def\Sp{\mathcal S}
\def\C{\mathcal C}
\def\Re{\mathcal R}
\def\R{\mathbb R}

\def\Z{\mathbb Z}
\def\A{\mathcal A}
\def\T1{T^1_{\{x_n\}}}
\def\Td1{T^{D,1}_{\{x_n\}}}
\def\Ts1{T^{S,1}_{\{x_n\}}}

\def\K{\mathcal{K}}

\def\Tid1{T^{D,1}_{\{x_n+\zeta_i\}}}
\def\Tis1{T^{S,1}_{\{x_n+\zeta_i\}}}
\def\Tjs1{T^{S,1}_{\{x_n+\zeta_j\}}}
\def\Tjd1{T^{D,1}_{\{x_n+\zeta_j\}}}

\newcount\colveccount
\newcommand*\colvec[1]{
        \global\colveccount#1
        \begin{pmatrix}
        \colvecnext
}
\def\colvecnext#1{
        #1
        \global\advance\colveccount-1
        \ifnum\colveccount>0
                \\
                \expandafter\colvecnext
        \else
                \end{pmatrix}
        \fi
}

\makeatletter
\newcommand{\labitem}[2]{%
\def\@itemlabel{\textbf{#1}}
\item
\def\@currentlabel{#1}\label{#2}}
\makeatother

\title{Classification for dynamics of Markov chains on non-negative integers with arbitrary transition rates and its application}
\author{Minjun Kim, Seokhwan Moon, Jinsu Kim}
\date{\today}

\begin{document}

\maketitle

\begin{abstract}
     Continuous-time Markov chains on non-negative integers can be used for modeling biological systems, population dynamics and queueing models. Qualitative behaviors of birth-and-death models, typical examples of such a one-dimensional continuous-time Markov chains, have been substantially studied. For one-dimensional Markov chains with polynomial transition rates, recent studies provided criteria for their long-term behavior. 
    In this paper, we provide a classification of Markov chains on non-negative integers when the transition rates are arbitrary functions. The criteria are written with asymptotics of the transition rates. This classification implies their dynamical properties, including explosivity, recurrence, positive recurrence, and exponential ergodicity. 
    As an application, we derive a complete classification (if and only if conditions) for those dynamical features when the transition rates have certain expansion forms, which include  all rational functions, so that our classifications cover mass-action kinetics, Michaelis–Menten kinetics, and Haldane equations. Our classification solely relies on easily computable quantities: the maximal degree of the expansion of the transition rates, the mean and the variance of the transition rates. We demonstrate the utility of this classification framework using the approximation of high-dimensional mass-action systems by a one-dimensional reaction system with general rational kinetics.
\end{abstract}

\section{Introduction}

One-dimensional continuous-time Markov chains on non-neg- ative integers have been widely used to model population growth in various biological and ecological systems, particularly describing the time evolution of copy numbers of proteins, bacteria, cells, and chemical species \cite{jia2020dynamical, holehouse2020stochastic, bo2017multiple}.
Thanks to their simplicity, these models are particularly useful and commonly applied. Typical examples include birth-and-death processes. They are sophisticated enough to serve as powerful tools for studying the copy numbers of biochemical species, yet simple enough to allow detailed analyses of their qualitative and quantitative behaviors.

However, when the Markov chain allows multiple jumps with arbitrary jump sizes, their qualitative properties such as positive recurrence and exponential ergodicity have not been fully characterized for general transition rates. In this paper, we provide simple criteria for determining various dynamical features of continuous-time Markov chains on non-negative integers. When the transition rates of these stochastic processes are rational functions, our criteria yield a complete classification based on their dynamical properties.

A critical motivation of our work is that one-dimensional Markov chains with non-polynomial transition rates can often be derived as reductions of higher-dimensional Markov chains governed by mass-action kinetics, which are fundamental polynomial laws for biochemical reaction systems. 
A typical example is the reduction of an enzyme--substrate model to the Michaelis--Menten kinetics model via a quasi--steady-state approximation \cite{segel1975enzyme}. 
Such time-scale separation--based model reductions can lead to more general one-dimensional Markov chains with non-polynomial transition rates. For instance, consider two biochemical species that catalyze each other under two distinct timescales described by the following biochemical reaction network:
\[
\begin{aligned}
&\emptyset \to 2A, \quad B + 2A \to B + A, \quad B + A \to B && \text{(slow reactions)},\\
&\emptyset \to B, \quad B + 2A \to 2A, \quad B \to \emptyset && \text{(fast reactions)}.
\end{aligned}
\]
Here, the scaled copy number of \( B \) rapidly converges, in probability, to values that can be expressed as a function of \( A \). 
Consequently, the dynamics of \( B \) can be approximated by a one-dimensional process involving the reactions \( A \to \emptyset \to 2A \), whose transition rates are rational functions (see Section~\ref{sec:reduction} for details). 
Due to such an utility of one-dimensional Markov models, previous studies focused qualitative behaviors of one-dimensional Markov models with non-polynomial transition rates \cite{lo2024strong, mao2004exponential, XuHansenWiuf2022}.

Several previous studies have investigated one-dimensional Markov chains obtained through such model reductions. 
In \cite{jia2020dynamical}, an autoregulated bursty gene expression model was analyzed via reduction to a one-dimensional Markov chain. 
In \cite{holehouse2020stochastic}, a reduced one-dimensional stochastic model for an enzyme--substrate system was used to derive detailed dynamical properties, including a closed-form solution of the chemical master equation and transient bimodality. 
Methods for reducing various biochemical systems to one-dimensional models were proposed in \cite{bo2017multiple}, with examples including bursty gene expression, receptor--ligand interactions, and the flagellar motor of \emph{E.~coli}. 
Furthermore, it was shown that a high-dimensional Markov chain can be bounded by a coupled one-dimensional Markov chain, and this coupling approach was applied to verify ergodicity of the original system \cite{ballif2025partition}. 
Notably, most of the resulting one-dimensional models in these studies admit non-polynomial transition rates. 

There also exist systematic approaches to model reduction for biochemical reaction systems. 
In \cite{song2021universally}, the validity of the stochastic quasi--steady-state approximation for a broad class of reaction systems was established. 
Model reduction techniques employing intermediate species were further developed in \cite{feliu2013simplifying, capellettie2016intermediate}.

Given the broad applicability of one-dimensional Markov models, understanding their long-term qualitative behavior is essential. 
Various properties, such as explosivity, recurrence, positive recurrence, and exponential ergodicity, have been analyzed for certain classes of one-dimensional population models, providing valuable insights into their stability and dynamics \cite{chen2004markov, mao2004exponential}. 
However, the applicability of these previous studies is limited. 
The criteria for ergodicity and exponential ergodicity proposed in \cite[Theorems~3.16, 4.52, and~4.54]{chen2004markov} and \cite{mao2004exponential} are valid only for unit-birth processes. 
Moreover, the criteria for exponential ergodicity given in \cite{mao2004exponential} are not necessary conditions unless the underlying process is a unit birth-and-death process. 
In other words, those results are confined to Markov chains where upward transitions are limited to births of size one.

For models with multiple birth sizes, criteria for various dynamical features were provided in \cite{XuHansenWiuf2022}. 
These criteria can be easily checked using the degrees of the transition rates and the corresponding transition vectors. 
Notably, in that work, the transition rates were restricted to polynomial functions. 
Therefore, the applicability of such results remains limited, since, as discussed earlier, many realistic biological models can be represented by one-dimensional Markov chains whose transition rates are rational or more general non-polynomial functions.

In this paper, we provide general criteria for determining the dynamical features of one-dimensional Markov chains on non-negative integers. 
Notably, our conditions are valid for arbitrary types of transition rates and for any finite set of possible jump sizes. 
The dynamical properties of interest include explosivity, transience, null recurrence, positive recurrence, and exponential ergodicity. 

We first establish general criteria on the arbitrary transition rates that can be used to verify these dynamical features. 
The key idea underlying our approach is to generalize the classical analysis of simple birth-and-death models, whose dynamical properties can be determined by elementary calculations involving the ratio of the birth and death rates. 
To achieve this generalization, we primarily employ Lyapunov function techniques. 

We then apply these general conditions to derive a complete classification of one-dimensional Markov chains whose transition rates $\lambda_i(x)$'s admit a \emph{Laurent-type expansion}: there exists $R\in \mathbb Z$ such that
\begin{align} 
    \lambda_i(x) = a_i x^R + b_i x^{R-1} + O(x^{R-2}) \quad \text{for each $i$}, \label{eq:laurent exp}
\end{align}
The exponent $R$ may be nonpositive, so this class of transition functions includes both polynomial and rational functions. 
Hence, our classification extends the previous results in \cite{XuHansenWiuf2022}. 
These classification criteria for the Laurent-type expansion (Theorem~\ref{thm: main for laurante}) can be checked using four easily computable indices, $\alpha$, $\gamma$, $\vartheta$, and $R$ (see~\eqref{eq : def-ctmc}), in a manner analogous to the criteria given in \cite{XuHansenWiuf2022}. 
Our classification is \emph{complete} in the sense that each dynamical feature corresponds one-to-one to a mutually disjoint subset of the index space $\mathbb{R} \times \mathbb{R}_{\ge 0} \times \mathbb{R}_{\ge 0} \times \mathbb{Z}$, to which the tuple $(\alpha, \gamma, \vartheta, R)$ belongs.

This paper is organized as follows. 
In Section~\ref{sec : motivation}, we provide the motivation for our main approach, which originates from unit birth-and-death models. 
Section~\ref{sec : prelim} presents the necessary preliminaries, including precise definitions of the dynamical features of interest and an overview of the Lyapunov function approach. 
In Section~\ref{sec: main}, we establish general criteria for each dynamical feature in terms of the mean and variance of the transitions. 
Section~\ref{sec : main-ctmc} states the main theorem for one-dimensional Markov chains whose transition rates satisfy \eqref{eq:laurent exp}. 
Finally, Section~\ref{sec : applications} illustrates applications of our classification to stochastic reaction systems.

\section{Motivation}\label{sec : motivation}
Our main theoretical tool for obtaining easily verifiable sufficient conditions for various dynamical features is the construction of Lyapunov functions. 
For a continuous-time Markov chain $\bar X$ on a countable state space $\bar \S$, its infinitesimal generator $\mathcal{A}$ is defined as
\begin{align}\label{eq:generator}
    \mathcal{A} f(x) = \sum_{y\in\bar \S} q_{x,y}f(y) \quad \text{for each $x \in \bar \S$},
\end{align}
for a suitable function $f$, where $q_{x,y}$ denotes the transition rate from $x$ to $y$. 
With a Lyapunov function $f$, the generator $\mathcal{A}$ can characterize long-term behaviors of $\bar X$, including recurrence, transience, and exponential ergodicity \cite{MT-LyaFosterIII,anderson2018some}. 
In particular, if there exists a positive function $f:\bar \S \to \mathbb{R}_{>0}$ such that i) $\mathcal{A} f(x) \le 0$ and ii) $f$ increases to infinity, then $\bar X$ is recurrent.

The construction of Lyapunov functions in this work is inspired by birth-and-death models, which are continuous-time Markov chains $\{X(t)\}_{t \ge 0}$ on $\mathbb{Z}_{\ge 0}$. 
For such processes, the transition rates are given by
\begin{align*}
    &P(X(t+\Delta t)=x+1 \mid X(t)=x) = b_x \Delta t + o(\Delta t),\\
    &P(X(t+\Delta t)=x-1 \mid X(t)=x) = d_x \Delta t + o(\Delta t),
\end{align*}
for sufficiently small $\Delta t$, where $b_x$ and $d_x$ are the birth and death rates, respectively, with $d_0 = 0$ to ensure that $X$ remains on the non-negative integers.

The following classical theorem provides a necessary and sufficient condition for recurrence of $X$ \cite{karlin1957classification}:
\begin{thm}[Theorem 1, \cite{karlin1957classification}]
    Let $X$ be a birth-and-death process with birth rates $b_x$ and death rates $d_x$. 
    Then $X$ is recurrent if and only if
    \begin{align}\label{eq:1}
        \sum_{x=1}^\infty \prod_{y=1}^x \frac{d_y}{b_y} = \infty.
    \end{align}
\end{thm}

From \eqref{eq:1}, the following asymptotic criteria are derived \cite[Theorem 1]{abramov2021necessary}:
\begin{align}
\begin{split}\label{eq:bd criteria}
  &\text{if } \frac{b_x}{d_x} \le 1 + \frac{1}{x} + \frac{1}{x \log x} \text{ for all but finitely many } x, \\
  &\text{ then $X$ is recurrent},\\
  &\text{if } \frac{b_x}{d_x} \ge 1 + \frac{1}{x} + \frac{q}{x \log x} \text{ for all but finitely many } x \text{ with some } q>1,\\ &\text{ then $X$ is transient}.
\end{split}
\end{align}

Alternatively, \eqref{eq:bd criteria} can be derived using the Foster-Lyapunov criteria. 
More precisely, for a birth-and-death process $X$, the generator in \eqref{eq:generator} can be expressed as
\begin{align}\label{eq:bd gen}
    \mathcal{A} f(x) = \frac{1}{d_x} \left( \frac{b_x}{d_x} \big(f(x+1)-f(x)\big) + \big(f(x-1)-f(x)\big) \right).
\end{align}
By the Foster-Lyapunov criteria for recurrence and transience (see Proposition~\ref{thm: hairer}), it suffices to construct a suitable function $f(x)$ such that \eqref{eq:bd gen} is nonpositive for sufficiently large $x$.  A natural choice is $f(x) \sim x^p (\log x)^q$ for some constants $p$ and $q$, which immediately yields \eqref{eq:bd criteria} by the inequality $\mathcal Af(x)<0$. 
Moreover, the same Lyapunov function of the form $f(x) \sim x^p (\log x)^q$ can be applied directly to one-dimensional Markov chains with multiple jump sizes to derive sharp criteria for explosivity, transience, null recurrence, and positive recurrence.
We further use the Lyapunov function $f(x) = e^{p x}$ for some constant $p$  to derive sufficient conditions only for exponential ergodicity. 
All the sufficient conditions are expressed in terms of the asymptotic behaviors of the transition rates, making them easily verifiable without additional analysis (see Remark \ref{rem : motivation}). 
We then show that these conditions are sharp enough to fully classify the dynamical behaviors of one-dimensional Markov chains whose transition rates satisfy \eqref{eq:laurent exp}, based on readily checkable indices associated with the transition rates.

\begin{rem}\label{rem : motivation}
In Section \ref{sec: main}, the sufficient conditions derived via Lyapunov functions for the dynamical properties of general one-dimensional Markov chains are expressed in terms of
\begin{align}\label{eq: h}
    H_p(x) := \frac{(\log x)(m(x)x - p v(x))}{v(x)}, \quad \text{and} \quad
    J(x) := \frac{m(x)x}{v(x)},
\end{align}
where $m(x)$ and $v(x)$ denote the mean and variance of the drift of $X$, respectively (see Definition \ref{defn:m and v}).

For birth-and-death processes, $H_1(x)$ emerges naturally when the classical recurrence and transience conditions are rewritten in terms of $m(x)$ and $v(x)$.  
In this case, $m(x) = b_x - d_x$ and $v(x) = \frac{b_x + d_x}{2}$, so that the upper part of \eqref{eq:bd criteria} can be rewritten as
\begin{align*}
    \frac{b_x}{d_x} = \frac{m(x) + 2 v(x)}{-m(x) + 2 v(x)} \le 1 + \frac{1}{x} + \frac{1}{x \log x}.
\end{align*}
Hence, for large $x$, we have
\begin{align}\label{eq: mv1}
    m(x) \le v(x) \left( \frac{1}{x} + \frac{1}{x \log x} \right),
\end{align}
which implies that $X$ is recurrent. 
Similarly, for transience, if there exists a constant $q > 1$ such that
\begin{align}\label{eq: mv2}
    m(x) \ge v(x) \left( \frac{1}{x} + \frac{q}{x \log x} \right)
\end{align}
for large $x$, then $X$ is transient. 

Each of the following inequalities implies the corresponding condition in \eqref{eq: mv1} and \eqref{eq: mv2}, respectively:
\begin{align*}
    \limsup_{x \to \infty} H_1(x) < 1, \quad \text{and} \quad \liminf_{x \to \infty} H_1(x) > 1.
\end{align*}

For other dynamical properties of more general one-dimensional Markov chains (i.e, multiple jump sizes), the Foster-Lyapunov approach allows the sufficient conditions to be expressed in terms of $m(x)$ and $v(x)$, and thus can also be represented using $H_p(x)$ and $J(x)$. 
This approach provides a unified framework to handle both single-jump (birth-and-death) and multiple-jump cases.

\end{rem}

\section{Preliminaries}\label{sec : prelim}

In this section, we provide precise definitions of the dynamical features of interest for continuous-time Markov chains and introduce the notations used throughout this paper.

\subsection{Main model}\label{sec:main model}

Throughout this paper, we consider an irreducible conti- nuous-time Markov chain $X$ on a countable state space $\mathbb S \subseteq \mathbb Z_{\ge 0}$ (possibly infinite) defined on a probability space $(\Omega, \mathcal F, P)$. The generator of $X$ is given by
\begin{align}\label{eq:gen for reactions}
    \mathcal A f(x) = \sum_{\eta \in \mathbb Z} \lambda_\eta(x) \big(f(x+\eta) - f(x)\big),
\end{align}
where for each transition $\eta \in \mathbb Z$, the corresponding transition rate $\lambda_\eta(x)$ satisfies
\begin{align*}
    P(X(t+\Delta t) = x+\eta \mid X(t)=x) = \lambda_\eta(x) \Delta t + o(\Delta t), \quad \text{for small } \Delta t.
\end{align*}
We restrict $X$ to non-negative integers so that it can represent the copy number of a species in a chemical reaction system, which is our main application. Accordingly, we set $\lambda_\eta(x) = 0$ for each $\eta < 0$ if $x < |\eta|$.

To study these Markov chains using Lyapunov function methods, we impose the following assumptions \cite{anderson2025new}.

\begin{assumption}\label{basic assumption}
The transition rates $\lambda_\eta(x)$ satisfy:
\begin{enumerate}[(a)]
    \item There exists a finite set $\Gamma \subset \mathbb Z$ such that $\lambda_\eta \equiv 0$ for all $\eta \notin \Gamma$.
    \item For each $x \in \mathbb S$, we have $0 \le \lambda_\eta(x) < \infty$ for all $\eta$.
\end{enumerate}
\end{assumption}

These conditions ensure that the number of possible jump sizes is finite. We note that, while some generalizations are possible \cite{anderson2025new}, the finiteness condition (a) is maintained throughout this paper.

\begin{rem}
Any countable state space $\bar{\mathbb S}$ can be enumerated and embedded in $\mathbb Z$. However, if the original Markov chain on $\bar{\mathbb S}$ has bounded transition vectors, such an enumeration may result in unbounded jump sizes. Thus, (a) of Assumption \ref{basic assumption} explicitly restricts attention to Markov chains on non-negative integers with bounded jump sizes, which are the main focus of this work.
\end{rem}

Throughout this manuscript, $X$ denotes an irreducible continuous-time Markov chain on $\mathbb S$ satisfying Assumption~\ref{basic assumption}. When citing theorems for Markov processes on a general state space, we denote the state space by $\bar{\mathbb S}$. In such cases, we denote the general Markov chain on $\bar{\mathbb S}$ by $\bar X$ for clarity.

\subsection{Notations and key concepts}

We introduce the key notations used throu- ghout this manuscript.

\begin{enumerate}
    \item $\mathbb Z_{\ge 0}$ denotes the set of non-negative integers.

    \item For a subset $D \subseteq \mathbb R$, a statement holds ``for large (resp. small) $x \in D$'' if there exists a constant $M>0$ such that the statement holds for all $x \in D$ with $x \ge M$ (resp. $|x| \le M$).

    \item Consider two real-valued functions $f,g: D\subseteq\mathbb R \to \mathbb R$ defined on the same domain $D$. We write $f(x) = O(g(x))$ as $x \to \infty$ ($x \to 0$) if there exists $C>0$ such that $|f(x)| \le C |g(x)|$ for large (small) $x \in D$. If $f(x) = O(g(x))$ and $g(x) = O(f(x))$ as $x \to \infty$ ($x \to 0$), we denote this by $f(x) \sim g(x)$ as $x \to \infty$ ($x \to 0$).

    \item Similarly, $f(x) \lesssim g(x)$ means that there exists $C>0$ such that $f(x) \le C g(x)$ for large $x \in D$. Note that the sign matters: for instance, if $f(x)=x$ and $g(x)=-x^2$, then $f(x) \not\lesssim g(x)$, while $f(x) = O(g(x))$ as $x \to \infty$.

    \item For a continuous-time Markov chain $\bar X$, let $T_k := \inf\{t \ge T_{k-1} : \bar X(t) \neq \bar X(T_{k-1})\}$ be the $k$th jump time with $T_0 = 0$. The \emph{explosion time} is defined as $T := \lim_{k \to \infty} T_k$.

    \item For each $x \in \bar \S$, the \emph{first return time} to $x$ is $\tau_x := \inf\{t \ge T_1 : \bar X(t) = x\}$.

    \item For two probability measures $\mu$ and $\nu$ on a discrete space $\bar \S$, the \emph{total variation distance} is $\|\mu - \nu\| := \sup_{A \subseteq \bar \S} |\mu(A) - \nu(A)|$.

    \item The \emph{embedded Markov chain} $\hat X$ of a continuous-time Markov chain $\bar X$ on $\bar \S$ is the discrete-time chain defined by $\hat X_k = \bar X(T_k)$ for each $k$.
\end{enumerate}

Next, we introduce the mean and variance of the transitions \cite{ lo2024strong, XuHansenWiuf2022}, which will be used to formulate our main criteria.
\begin{defn}\label{defn:m and v}
    For a Markov chain $X$ defined as in Section \ref{sec:main model}, the mean transition (drift) and the variance of the transitions are
    \begin{align*}
        m(x) = \sum_{\eta} \eta \, \lambda_\eta(x), \quad 
        v(x) = \frac{1}{2} \sum_{\eta} \eta^2 \, \lambda_\eta(x).
    \end{align*}
\end{defn}

We also introduce the following functions to express the classification criteria:
\begin{defn} $H_p(x) := \dfrac{(\log x)(m(x)x - p v(x))}{v(x)}, \quad \text{and} \quad 
        J(x) := \dfrac{m(x)x}{v(x)}$, for some constant $p$.
\end{defn}
Note that these functions are exactly the same as \eqref{eq: h}.
\subsection{Dynamical features of CTMC}

The main goal of this paper is to provide sufficient conditions for various dynamical properties of a continuous-time Markov chain defined in Section \ref{sec:main model}. 
The dynamical properties we consider include explosivity, non-explosivity, transience, recurrence, null-recurrence, positive recurrence, non-exponential ergodicity, and exponential ergodicity. 
While these properties are elementary, we provide precise definitions below.

\begin{defn}
    Let $\bar X$ be an irreducible continuous-time Markov chain on $\bar{\mathbb S}$. Then 
    \begin{enumerate}
         \item $\bar X$ is explosive if $P(T<\infty)>0$. Otherwise, $\bar X$ is non-explosive.
         \item $\bar X$ is recurrent if $P_x(\tau_x<\infty)=1$ for some $x \in \bar{\mathbb S}$. Otherwise, $\bar X$ is transient.
         \item $\bar X$ is positive recurrent if $\mathbb E_x(\tau_x)<\infty$ for some $x \in \bar{\mathbb S}$. If $\bar X$ is recurrent but not positive recurrent, it is null-recurrent. 
         \item $\bar X$ is exponentially ergodic if it is ergodic and there exist constants $C(x)>0$ and $\eta>0$ such that
         \begin{align*}
             \Vert P_x(\bar X(t)\in \cdot)-\pi(\cdot)\Vert_{TV}\le C(x)e^{-\eta t} \quad \text{for all $\bar X(0)=x$.}
         \end{align*}
         If $\bar X$ is ergodic but not exponentially ergodic, it is non-exponentially ergodic. 
    \end{enumerate}
\end{defn}

These dynamical features imply important statistical properties of $\bar X$. 
For example, for an irreducible $\bar X$ on a countable state space, $\bar X$ is positive recurrent if and only if there exists a unique \textit{stationary probability distribution} $\pi$.

\begin{defn}
    For a continuous-time Markov chain $\bar X$ on a countable state space $\bar \S$, $\pi$ is a stationary measure if 
    \[
        P_\pi(\bar X(t)=x):=\sum_{y} P_y(\bar X(t)=x)\pi(y)=\pi(x)
    \]
    for all $x\in \bar{\mathbb S}$ and $t\ge0$. 
    If $\sum_{x}\pi(x)=1$, then $\pi$ is a stationary distribution. 
\end{defn}

\noindent The one-step analysis shows that a stationary measure $\pi$ is a steady solution of the \textit{master equation}:
\begin{align*}
    \sum_{y}q_{y,x}\pi(y)-\sum_y q_{x,y} \pi(x)=0 \quad \text{for each $x \in \bar \S$}
\end{align*}

\noindent Furthermore, an irreducible, positive recurrent $\bar X$ on $\bar{\mathbb S}$ is \textit{ergodic} \cite{norris1997jr}, meaning that
\[
    \lim_{t \to \infty} \Vert P_x(\bar X(t)\in \cdot)-\pi(\cdot)\Vert_{tv}=0 \quad \text{for any $\bar X(0)=x$}.
\]

\subsection{Lyapunov conditions for the dynamical features}

We present here the Lyapunov functions used to derive our main results. 
Foster-Lyapunov criteria \cite{MT-LyaFosterI, MT-LyaFosterII, MT-LyaFosterIII} imply that if a function $f:\bar \S\to \mathbb R_{\ge 0}$ catches a sufficiently negative mean drift of $f(\bar X)$, then $\bar X$ possesses certain dynamical properties. Negative drifts of $f(\bar X)$ can be captured via $\mathcal A f$. 
As introduced in \eqref{eq:generator}, the generator acts on a function suitable function $f$, that is, $f \in \text{Dom}_+(\A)$ as
\begin{align*}
    \A f(x) = \sum_{y \in \mathbb S} {q_{x,y}}f(y) \quad \text{for each $x \in \bar \S$.}
\end{align*}
 Dom$_+(\A)$ is a collection of positive functions such that $\sum_{y\neq x}q_{x,y}|f(y)|< \infty$ for each $x \in \bar \S$ \cite[Section 1.2]{menshikov2014explosion}. In case that $X$ satisfies Assumption \ref{basic assumption}, every positive real-valued function defined on the state space $\S$ belongs to $\text{Dom}_+(\A)$, and we have that 
\begin{align*}
    \A f(x) = \sum_{\eta \in \Gamma} \lambda_{\eta}(x)(f(x+\eta)-f(x)). 
\end{align*} 

As our main result, we derive sufficient conditions for the dynamical properties 
of continuous-time Markov chains $X$ on $\mathbb Z_{\ge 0}$ using Lyapunov functions. These conditions are expressed solely in terms of the mean drift $m(x)$ and variance $v(x)$ defined in Definition \ref{defn:m and v}. 
We consider two types of Foster-Lyapunov functions: $(i)$ $x^p(\log x)^q$ and $(ii)$ $e^{px}$.
The infinitesimal generators acting on these functions can be represented in terms of $m(x)$ and $v(x)$ as follows:

\begin{prop}\label{prop: generators}
   Let $f$ be a function defined on $\S \subseteq\mathbb Z_{\ge 0}$. 
    \begin{enumerate}
        \item If $f(x) = x^p (\log x)^q \mathbbm{1}_{\{x > 1\}}+\mathbbm{1}_{\{x\le 1\}}$, then as $x \to \infty$, 
        \begin{align*}
            &\A f(x) = p x^{p-2}(\log x)^q \left\{m(x)x + (p-1)v(x)+O\left(\frac{v(x)}{x}\right)\right\}  \\ 
            & \quad  + q x^{p-2}(\log x)^{q-1}\left\{m(x) x + (2p-1)v(x) + (q-1)\frac{v(x)}{\log x}+O\left(\frac{v(x)}{x}\right)\right\}
        \end{align*}
        \item If $f(x) = e^{px}$, then for $x \in \mathbb S$,
        \begin{align*}
        e^{p x} \{pm(x) + p^2 v(x) + O(p ^3)O( v(x))\} \quad \text{as} \quad p \to 0,\ x\to \infty.
    \end{align*}
     \end{enumerate}
\end{prop}

The following corollary provides a convenient consequence of this expansion.

\begin{cor}\label{cor: for the sup assumption}
    For an irreducible continuous-time Markov chain on $\mathbb S\subseteq \mathbb Z_{\ge 0}$ satisfying assumption \ref{basic assumption}, 
    \begin{enumerate}
        \item if $\dlimsup_{x\to \infty}H_1(x)<1$, then there exists a positive function $f$ on $\mathbb S$ and $\epsilon>0$ such that $\lim_{x \to \infty}f(x) = \infty$, $\A f(x) < 0$ for all large $x$, and  
    \begin{align}\label{eq: limsup1}
        \dlimsup_{x\to \infty}\mathcal A f(x) \le -\left(\frac{\epsilon}{2} \right) ^2\liminf_{x\to \infty}\frac{v(x)}{x^2 (\log x)^{2-(\epsilon/2)}} \le 0
    \end{align}
    \item if $\dliminf_{x\to \infty}H_1(x)>1$, then there exists a positive function $f$ on $\mathbb S$ and $\epsilon>0$ such that $\lim_{x \to \infty}f(x) = 0$, $\A f(x) < 0$ for all large $x$, and  
    \begin{align}\label{eq: limsup2}
        \dlimsup_{x\to \infty}\mathcal A f(x) \le -\left(\frac{\epsilon}{2} \right) ^2\dliminf_{x\to \infty}\frac{v(x)}{x^2 (\log x)^{2+(\epsilon/2)}} \le 0
    \end{align}
    \end{enumerate}
\end{cor}

See Appendix~\ref{sec : generator-expansion} for  proofs of Proposition \ref{prop: generators} and Corollary \ref{cor: for the sup assumption}.

\section{Main results}\label{sec: main}
The main sufficient conditions for various dynamical features of one-dimensional continuous-time Markov chains can be expressed in terms of the asymptotic behaviors of the drift mean $m(x)$ and variance $v(x)$ introduced in Definition~\ref{defn:m and v}.
The results are summarized below.
\begin{thm}\label{thm:main}
    Let $X$ be an irreducible continuous-time Markov chain on $\mathbb S \subseteq \mathbb Z_{\ge 0}$ satisfying Assumption~\ref{basic assumption}.
Let $m(x)$ and $v(x)$ denote the drift mean and variance of $X$ as defined in Definition~\ref{defn:m and v}.
Then the following hold:
    \begin{enumerate}
        \item $X$ is explosive if
        \begin{enumerate}
            \item $x\cdot m(x)-v(x) \gtrsim x^c$ for some $c>2$, and
            \item   $\dliminf_{x \to \infty}H_1(x) >1$,
            \item  $\dliminf_{x \to \infty}\frac{v(x)}{x^2 (\log x)^q} = 0$ implies $\dlim_{x \to \infty}\frac{v(x)}{x^2 (\log x)^q} = 0$ for any $q>2$. 
        \end{enumerate}
            \item $X$ is non- explosive if either
           \begin{enumerate}
               \item  $x\cdot m(x)-v(x)\lesssim x^2\log x$ and $v(x) \lesssim x^2(\log x)^2$, or \item  $\dlimsup_{x \to \infty}H_1(x)<1$.
           \end{enumerate} 
           \item $X$ is transient if $\dliminf_{x\to \infty}H_1(x)>1$.
            \item $X$ is recurrent if $\dlimsup_{x\to \infty} H_1(x)<1$,
            \item $X$ is positive recurrent if either
            \begin{enumerate}
                \item $v(x) \gtrsim x^c$ for some $c \le 2$ and  $
        \dlimsup_{x \to \infty}J(x) <(c-1)$, or
    \item $v(x) \gtrsim x^c$ for some $c>2$ and $
        \dlimsup_{x\to \infty} H_1(x)<1$.
            \end{enumerate}
            \item When $X$ is recurrent, $X$ is null recurrent if there exists a constant $0<c\le 2$ such that $v(x) \lesssim x^{c}$ and either 
    \begin{enumerate}
        \item
         $c-1 < \dliminf_{x \to \infty} J(x) \le \dlimsup_{x \to \infty} J(x)< \infty$, or
    \item 
   $-1 < \dliminf_{x\to \infty} H_{c-1}(x)\le \dlimsup_{x\to \infty} H_{c-1}(x)< \infty$
    \end{enumerate}

\item 
Assume that $X$ is positive recurrent. $X$ is exponentially ergodic if either 
    \begin{enumerate}
        \item there exists a constant $c>2$ such that $v(x) \gtrsim x^{c}$ and 
       $ \dlimsup_{x\to \infty} H_1(x)<1$,
        \item  $v(x) \sim x^2$ and $\dlimsup_{x \to \infty} J(x)<1$, or
        \item $1 \lesssim v(x)$ and $\dlimsup_{x \to \infty} J(x)/x<0$.
    \end{enumerate}

\item Assume that $X$ is positive recurrent.
$X$ is non-exponentially ergodic if either
        \begin{enumerate}
            \item $v(x) \lesssim x^c$ for some $c<2$ and $-\infty  < \dliminf_{x \to \infty}J(x) \le \dlimsup_{x \to \infty}J(x) < \infty$, or
            \item $v(x) \lesssim x^c$ for some $c <0$.
        \end{enumerate}   
    \end{enumerate}
\end{thm}

In the following sections, we provide detailed proofs of each case.
For every dynamical property, we begin with the corresponding Foster–Lyapunov criterion, which serves as the main analytical tool.
Note that these Foster–Lyapunov criteria are formulated for general continuous-time Markov chains $\bar X$ on general countable state spaces $\bar{\mathbb S}$, whereas Theorem~\ref{thm:main} concerns Markov chains $X$ on $\mathbb S$.

\subsection{Explosivity}
\begin{prop}[\cite{menshikov2014explosion}, Theorem 1.12, Remark 1.13]\label{thm:explosion_ref}
    For an irreducible Markov chain $\bar X$ on a countable state space $\bar{\mathbb S}$, Suppose that there exist $\delta>0$, a finite subset $ A \subseteq \bar{\mathbb S}$, and $f \in \text{Dom}_+(\A)$ such that
    \begin{enumerate}[(i)]
      \item there exists $x_0 \in A^c$ such that $f(x_0) < \dmin_{x \in A}f(x)$,
        \item $\A f(x) \le -\delta$ for each $x\in A^c$. 
    \end{enumerate}
    Then $E(T)<\infty$, and in particular, $X$ is explosive.
\end{prop}

Using $f(x)\sim (\log x)^p$ with $p<0$, we derive the following sufficient condition for explosion.
\begin{thm}[Explosivity]\label{thm: expl}
$X$ is explosive if
        \begin{enumerate}
            \item[(a)] $m(x)x-v(x) \gtrsim x^c$ for some $c>2$, 
            \item[(b)]   $\dliminf_{x \to \infty}H_1(x) >1$, and
            \item [(c)] $\dliminf_{x \to \infty}\frac{v(x)}{x^2 (\log x)^q} = 0$ implies $\dlim_{x \to \infty}\frac{v(x)}{x^2 (\log x)^q} = 0$ for any $q>2$.
        \end{enumerate}
\end{thm}
\begin{proof}
By Proposition \ref{thm:explosion_ref}, it suffices to find a positive function $f$ on $\S$ such that $\dlim_{x \to \infty} f(x) = 0$ and $\dlimsup_{x \to \infty}\A f(x)<0$. We take $f(x) = (\log x)^{-\epsilon/2}\mathbbm{1}_{\{x>1\}}+\mathbbm{1}_{\{x\le 1\}}$ for some $\epsilon>0$, as in the proof of part (2) of Corollary~\ref{cor: for the sup assumption}.
If
$\dliminf_{x\to \infty} \frac{v(x)}{x^2(\log x)^{2+(\epsilon/2)}} > 0$, the desired result follows from Corollary~\ref{cor: for the sup assumption}.

Thus, it remains to consider the case $\dliminf_{x\to \infty} \frac{v(x)}{x^2(\log x)^{2+(\epsilon/2)}} = 0$, which implies $\dlim_{x\to \infty} \frac{v(x)}{x^2(\log x)^{2+(\epsilon/2)}} = 0$.
By Proposition~\ref{prop: generators} with $p=0$ and $q=\epsilon/2$, we have that, as $x \to \infty$,
\begin{align}
    \A f(x) &= \frac{-\epsilon/2}{x^2(\log x)^{(\epsilon/2)+1}}\left\{m(x)x-v(x) - (1+ \epsilon/2)\frac{v(x)}{\log x}+O\left(\frac{v(x)}{x}\right)\right\} \\ 
    &= -\frac{\epsilon}{2}\left\{\frac{m(x)x-v(x)}{x^2(\log x)^{(\epsilon/2)+1}}-(1+\epsilon/2) \frac{v(x)}{x^2(\log x)^{2+(\epsilon/2)}}\right\}\label{eq: expl 1} 
\end{align}

Then $\dlim_{x\to \infty} \frac{v(x)}{x^2(\log x)^{2+(\epsilon/2)}} = 0$ implies
\begin{align*}
    \dlimsup_{x \to \infty} \A f(x) = -\frac{\epsilon}{2}\liminf_{x \to \infty}\frac{m(x)x-v(x)}{x^2(\log x)^{(\epsilon/2)+1}} = -\infty,
\end{align*}
where the last equality follows from assumption $(a)$.
\end{proof}

For non-explosivity, we use the following result.
\begin{prop}[\cite{menshikov2014explosion}, Theorem 1.14]\label{thm: non-expl}
For an irreducible Markov chain $\bar X$ on a countable state space $\bar{\mathbb S}$, suppose there exists  $f \in \text{Dom}_+(\A)$ such that 
    \begin{enumerate}[(i)]
        \item $\dlim_{x \to \infty}f(x) = \infty$,
        \item there exists an increasing (not necessarily strict) function $g: \R_{> 0} \to \R_{> 0}$ such that $G(z): = \int_0^z \frac{dy}{g(y)}<\infty$ for each $z\in \bar \S$ and $\dlim_{z \to \infty}G(z) = \infty$, and
        \item $\A f(x) \le g(f(x))$ for all $x \in \bar \S$. 
    \end{enumerate}
    Then $P_x(T = \infty) = 1$ for all $x \in \bar \S$.
\end{prop}

Considering $g(x)$ as a constant or linear function, we have the following theorem for non-explosivity with $f(x) \sim (\log x)^p$.
\begin{thm}[Non-explosivity]\label{thm: non-expl 2}
    $X$ is non-explosive if either
           \begin{enumerate}[(a)]
               \item  $m(x)x-v(x)\lesssim x^2\log x$ and $v(x) \lesssim x^2(\log x)^2$, or \item  $\dlimsup_{x \to \infty}H_1(x)<1$.
           \end{enumerate} 
\end{thm}
\begin{proof}

If assumption~$(b)$ holds, then we can take $f$ as in the proof of Corollary~\ref{cor: for the sup assumption} and set $g \equiv c>0$, so that the conditions in Proposition~\ref{thm: non-expl} follow from Corollary~\ref{cor: for the sup assumption}.

Now assume $(a)$.
Let $f=(\log x)\mathbbm{1}_{\{x > 1\}}+\mathbbm{1}_{\{x\le 1\}}$, which satisfies condition (i) of Proposition~\ref{thm: non-expl}. By Proposition \ref{prop: generators}, we have
\begin{align}
    \A f(x) &= \frac{1}{x^2}\left\{m(x)x -v(x) + O\left(\frac{v(x)}{x}\right)\right\}\quad \text{as $x \to \infty$}\label{eq: non-expl 1}.
\end{align}
From assumption $(a)$ and \eqref{eq: non-expl 1}, it follows that $\A f(x) \le af(x)+b$ for all $x \in \mathbb S$ with some positive constants $a$ and $b$. Then conditions $(ii)$ and $(iii)$ of Proposition \ref{thm: non-expl} hold with $g(x) = ax+b$, completing the proof.
\end{proof}

\subsection{Recurrence}

For transience and recurrence, we provide a version of Foster–Lyapunov criteria obtained by combining \cite[Proposition 1.3]{hairer2010convergence} and \cite[Theorem 3.4.1]{NorrisMC97}:

\begin{prop}\label{thm: hairer}
    Let $\bar X$ be an irreducible continuous-time Markov chain on a countable state space $\bar{\mathbb S}$. Then 
    \begin{enumerate}[(i)]
        \item $\bar X$ is transient if and only if there exists a function $f\in \text{Dom}_+(\A)$ and a non-empty set $A \subset \bar \S$ such that $\A f(x) \le 0$ for all $x \not \in A$ and there exists $x_0 \not \in A$ such that $f(x_0) < \inf_{y \in A}f(y)$. 
        \item $\bar X$ is recurrent if and only if there exists a function $f\in \text{Dom}_+(\A)$, which satisfies $|\{x: f(x) \le N \}| < \infty$ for every $N>0$ and $\A f(x)\le 0$ for all but finitely many $x$. 
    \end{enumerate}
\end{prop}

Applying these criteria to $X$, the additional conditions on $f$ can be simplified, leading to the following corollary.
    \begin{cor}\label{cor: rec}
        $X$ is recurrent (transient, respectively) if there exists a positive function $f$ on $\S$ such that $\dlim_{x \to \infty}f(x) = \infty \ (0, \text{respectively})$ and $\mathcal Af(x) \le 0$ for all large $x$. 
    \end{cor}
    Further applications yield more concrete results.
\begin{thm}[Recurrence]\label{thm: recu}
    $X$ is recurrent if 
        $\dlimsup_{x \to \infty}H_1(x) <1$.
\end{thm}
\begin{proof}
    The conditions of Corollary~\ref{cor: rec} follow from (1) of Corollary~\ref{cor: for the sup assumption}.
\end{proof}
\begin{thm}[Transience]\label{thm: tran}
         $X$ is transient if 
        $\dliminf_{x \to \infty}H_1(x) > 1$.
\end{thm}
\begin{proof}
 The conditions of Corollary~\ref{cor: rec} follow from (2) of Corollary~\ref{cor: for the sup assumption}.
\end{proof}

\subsection{Positive recurrence and null recurrence}

\cite{MT-LyaFosterIII} established Foster–Lya- punov criteria for positive recurrence of Markov processes on general topological spaces. Here, we introduce a restricted version applicable to continuous-time Markov chains on countable state spaces, which is our main focus.

\begin{prop}[\text{\cite[Proposition 11]{XuHansenWiuf2022}}]\label{thm: mt}
Let $\bar X$ be an irreducible, recurrent continuous-time Markov chain on countable state space $\bar \S$. Then the following are equivalent: 
\begin{enumerate}
   \item  $\bar X$ is positive recurrent. 
    \item There exist $\delta>0$, a finite subset $A \subseteq \bar \S$, and $f \in \text{Dom}_+(\A)$ satisfying $\A f(x) \le -\delta$ for all $ x\in \bar \S\setminus A$. 
\end{enumerate}
\end{prop}

By applying this proposition to $X$ and using Corollary~\ref{cor: rec}, we obtain the following simplified sufficient condition.

\begin{cor}
    $X$ is positive recurrent if there exists a positive function $f$ on $\S$ such that $\dlim_{x \to \infty}f(x) = \infty$ and $\dlimsup_{x \to \infty}\mathcal Af(x) < 0$. 
\end{cor}
The following theorem provides sufficient conditions on $m(x)$ and $v(x)$ for this criterion to hold.
\begin{thm}[positive recurrence]\label{thm: positive}
    $X$ is positive recurrent if either
\begin{enumerate}[(a)]
                \item $v(x) \gtrsim x^c$ for some $c \le 2$ and  $
        \limsup_{x \to \infty}J(x) <(c-1)$, or
    \item $v(x) \gtrsim x^c$ for some $c>2$ and $
        \limsup_{x\to \infty} H_1(x)<1$.
            \end{enumerate}
\end{thm}

\begin{proof}
When (b) holds, Corollary~\ref{cor: for the sup assumption} implies the existence of a positive function $f$ on $\S$ such that $\dlim_{x \to \infty} f(x) = \infty$ and $\dlimsup_{x \to \infty} \A f(x) = -\infty$. 

Now assume (a). Then there exists $\epsilon > 0$ such that
\begin{align}\label{eq: ergo}
        \limsup_{x \to \infty}J(x) \le c-1-\epsilon. 
    \end{align}
For this $\epsilon$, define a positive function $f$ on $\S$ by $f(x) = x^{2 - c + \epsilon/2}$, which increases to infinity as $x \to \infty$. 
By Proposition~\ref{prop: generators}, as $x \to \infty$,
\begin{align*}
            \mathcal A f(x) &=(2-c+\epsilon/2)\frac{x^{\epsilon/2}}{x^c}\left\{m(x)x+(1-c+\epsilon/2)v(x)+O\left(\frac{v(x)}{x}\right)\right\} \\ &=(2-c+\epsilon/2)\frac{x^{\epsilon/2}v(x)}{x^c}\left\{J(x)-(c-1-\epsilon) - \epsilon/2 +O\left(\frac{1}{x}\right)\right\}.
    \end{align*}
Hence, by \eqref{eq: ergo} and $v(x) \gtrsim x^c$, we have $\dlimsup_{x \to \infty} \A f(x) = -\infty$.
\end{proof}

Now we turn to null recurrence. We begin with a result on the tail behavior of stationary measures for discrete-time Markov chains $\hat X$ on countable state spaces $\hat \S$. A measure $\nu$ on $\hat \S$ is a \textit{stationary measure} of $\hat X$ if $P_\nu(\hat X_n = x) := \sum_y P_y(\hat X_n = x)\nu(y)$ for each $x \in \hat \S$ and each $n$. 
The existence and uniqueness of such $\nu$ (up to a multiplicative constant) are guaranteed when $\hat X$ is irreducible and recurrent. 
\begin{prop}[\cite{XuHansenWiuf2022}, Proposition 16]\label{prop: null}
 Let $\hat X$ be an irreducible, recurrent discrete-time Markov chain on $\hat \S$
 with a unique (up to a multiplicative constant) stationary measure \(\nu\).
Let  $g$  be a nonnegative function defined on $\hat \S$.
Then
\begin{align}\label{eq:null recurrence, infinity}
    \sum_{x \in \hat \S} g(x) \nu(x) = \infty
\end{align}
if there exist a finite set \(A \subseteq \hat \S\), a state \(x_0 \in A\) 
, and non-negative functions \(f_1\) and \(f_2\) such that

\begin{enumerate}[(i)]
    \item \( \dlim_{x \to \infty} f_1(x) = \infty \) and \( \dlim_{x \to \infty} \dfrac{f_2(x)}{f_1(x)} = \infty \),
    \item  $E_{x}[f_1(\hat X_{1})]\ge f_1(x)$ for each $x \in \hat \S \setminus A$,
    \item \( \mathbb{E}_{x_0} [ f_2(\hat X_n) \mathbf{1}_{\hat \tau_A > n} ] \) is finite for all \( n \) where $\hat \tau_A=\inf\{n\ge 1: \hat X_n\in A\}$, and
    \item  \( \mathbb{E}_{x}[f_2(\hat X_{1})] \leq g(x) + f_2(x) \) for each $x \in \hat\S \setminus A$.
\end{enumerate}

\end{prop}

When $X$ is recurrent, its embedded chain $\hat X$ is also recurrent and admits a unique stationary measure $\nu$ up to a multiplicative constant. Then $\pi(x) := g(x)\nu(x)$ is the stationary measure of $X$, since $P_\pi(X(t) = x) = \pi(x)$ for each $x$ and $t$, where $g(x) = \dfrac{1}{\sum_\eta \lambda_\eta(x)}$. 
Hence, to establish null recurrence of $X$, it suffices to verify \eqref{eq:null recurrence, infinity}.  

Condition $(iii)$ of Proposition~\ref{prop: null} follows from the boundedness of the jump set $\Gamma$. Moreover, for any $h \in \text{Dom}_+(\A)$,
 \begin{align*}
     E_x[h(\hat X_1)]-h(x) = g(x)\mathcal A h(x) \quad \text{for all $x \in \mathbb S$.}  
 \end{align*}
  Thus, conditions $(ii)$ and $(iv)$ of Proposition~\ref{prop: null} can be rewritten in terms of the infinitesimal generator of $X$. This yields the following corollary.
\begin{cor}\label{cor: null}
    When $X$ is recurrent, $X$ is null recurrent if there exists non-negative functions $f_1,f_2$ such that
    \begin{enumerate}[(a)]
        \item [(i)] $\dlim_{x \to \infty} f_1(x) = \infty$ and $\dlim_{x \to \infty} \dfrac{f_2(x)}{f_1(x)} = \infty$,
        \item [(ii)] $\A f_1(x) \ge 0$ for large $x$, and
        \item [(iii)] $\A f_2(x) \lesssim 1$.
    \end{enumerate}
\end{cor}

We now derive simple conditions on $m(x)$ and $v(x)$ for null recurrence of $X$.
\begin{thm}[Null recurrence]\label{thm: null}
    When $X$ is recurrent, $X$ is null recurrent if there exists a constant $0<c\le 2$ such that $v(x) \lesssim x^{c}$ and either 
    \begin{enumerate}[(a)]
        \item
         $c-1 < \dliminf_{x \to \infty} J(x) \le \dlimsup_{x \to \infty} J(x)< \infty$, or
    \item 
   $-1 < \dliminf_{x\to \infty} H_{c-1}(x)\le \dlimsup_{x\to \infty} H_{c-1}(x)< \infty$.
    \end{enumerate}
\end{thm}
\begin{proof}
    We consider two cases, $c = 2$ and $c < 2$, and under each case construct positive functions $f_1, f_2$ on $\S$ satisfying conditions $(i)$–$(iii)$ of Corollary~\ref{cor: null}.

       \vspace{0.5cm}
    \noindent \textbf{Case 1: $c = 2$.} Define $f_i(x) = (\log x)^{q_i}\mathbbm{1}_{\{x > 1\}} + \mathbbm{1}_{\{x \le 1\}}$ for $i=1,2$, where $0 < q_1 < q_2$. Then condition $(i)$ holds.  
By Proposition~\ref{prop: generators}, for each $i$ as $x \to \infty$,
\begin{align}
        \mathcal A f_i(x) &=\frac{q_i(\log x)^{q_i-1}}{x^2}\left\{m(x)x -v(x)+(q_i-1)\frac{v(x)}{\log x}+O\left(\frac{v(x)}{x}\right)\right\} \notag\\ &= \frac{q_i(\log x)^{q_i-1}v(x)}{x^2}\left\{J(x) -1 +\frac{q_i-1}{\log x}+O\left(\frac{1}{x}\right)\right\} \label{eq: null 1} \\ &=\frac{q_i(\log x)^{q_i-2}v(x)}{x^2}\left\{H_1(x)+(q_i-1)+O\left(\frac{\log x}{x}\right)\right\} \label{eq: null 2}
        \end{align}
If $(a)$ holds, since $v(x) \lesssim x^2$, we have $\A f_i(x) \ge 0$ for large $x$ and $\A f_i(x) \lesssim (\log x)^{q_i - 1}$. Thus, taking $0 < q_1 < q_2 \le 1$, conditions $(ii)$ and $(iii)$ follow.  

If $(b)$ holds, choose $0 < q_1 < q_2 \le 2$ such that
\begin{align}\label{eq: null 3}
    1-q_2<1-q_1<\liminf_{x\to \infty} H_{1}(x)\le\limsup_{x \to \infty}H_{1}(x)<\infty.
\end{align}
Then $v(x) \lesssim x^2$ together with \eqref{eq: null 2} and \eqref{eq: null 3} implies $\A f_i(x) \ge 0$ for large $x$ and $\A f_i(x) \lesssim (\log x)^{q_i - 2} \lesssim 1$. Hence $(ii)$ and $(iii)$ follow.

\vspace{0.5cm}
\noindent \textbf{Case 2: $0<c < 2$.} Assume $(a)$. Then there exists $0 < \epsilon < 2 - c$ such that $c - 1 + \epsilon \le \dliminf_{x \to \infty} J(x)$. For $i=1,2$, define $f_i(x) = x^{p_i}\mathbbm{1}_{\{x > 1\}} + \mathbbm{1}_{\{x \le 1\}}$, where $p_1 = 2 - c - \epsilon/2$ and $p_2 = 2 - c$. Then $(i)$ holds, and by Proposition~\ref{prop: generators},
\begin{align}\label{eq: x f 2}
    \A f_i(x)=p_i\frac{x^{p_i}v(x)}{x^2}\left\{J(x)+p_i-1+O\left(\frac{1}{x}\right)\right\}.
\end{align}
Since $v(x) \lesssim x^c$ and by the construction of $p_i$, taking limits yields conditions $(ii)$ and $(iii)$.  

Now assume $(b)$. For $i=1,2$, let $f_i(x) = x^{2 - c} (\log x)^{q_i}\mathbbm{1}_{\{x > 1\}} + \mathbbm{1}_{\{x \le 1\}}$, with $q_2 > q_1 > 0$. Then $(i)$ holds, and by Proposition~\ref{prop: generators},
        \begin{align}
            \mathcal Af_i(x)  &=(2-c)x^{-c}v(x)(\log x)^{q_i-1}\left\{H_{c-1}(x)+O\left(\frac{\log x}{x}\right)\right\}\notag\\&+q_ix^{-c}v(x)(\log x)^{q_i-2}\left\{H_{c-1}(x)+q_i-1+(2-c){\log x}+O\left(\frac{\log x}{x}\right)\right\}\notag \\ =&\frac{v(x)}{x^c}(\log x)^{q_i-1}[(2-c)\{H_{c-1}(x)+q_i\}\notag\\
            &~~~~~~~~~~~~~~~~~~~~~+\frac{q_i}{\log x}\{H_{c-1}(x)+q_i-1\}+O\left(\frac{\log x}{x}\right)], \label{eq: null}
        \end{align}
        By assumption $(b)$, we can take $0 < q_1 < q_2 \le 1$ so that 
        \begin{align}\label{eq: null 6}
            -q_2<-q_1<\liminf_{x \to \infty}H_{c-1}(x)\le \limsup_{x \to \infty}H_{c-1}(x) < \infty.
        \end{align}
Then $v(x) \lesssim x^c$, together with \eqref{eq: null} and \eqref{eq: null 6}, implies $\A f_i(x) \ge 0$ for large $x$ and $\A f_i(x) \lesssim (\log x)^{q_i - 1} \lesssim 1$. Thus conditions $(ii)$ and $(iii)$ of Corollary~\ref{cor: null} follow.

\end{proof}

\subsection{Exponential ergodicity}\label{sec:expo}
We apply the following Foster-Lyapunov criterion for exponential ergodicity. 
\begin{prop}[\cite{XuHansenWiuf2022}, Proposition 12]\label{prop: expo}
    Let $\bar X$ be an irreducible continuous-time Markov chain on $\bar \S$. Then $\bar X$ is exponentially ergodic if there exist $C>0$, a finite subset $A \subseteq \bar \S$, and $f \in \text{Dom}_+(\A)$ that increases to infinity such that 
    \begin{align*}
        \A f(x) \le -C f(x) \quad \text{for} \quad x \in \bar \S \setminus A. 
    \end{align*}
\end{prop}

Using this proposition, we derive criteria for exponential and non-exponential ergodicity. 
\begin{thm}[Exponential ergodicity]\label{thm: expo}
        Assume that $X$ is positive recurrent. Then $X$ is exponentially ergodic if either
    \begin{enumerate}[(a)]
        \item there exists a constant $c>2$ such that $v(x) \gtrsim x^{c}$ and 
       $ \dlimsup_{x\to \infty} H_1(x)<1$,
        \item  $v(x) \sim x^2$ and $\dlimsup_{x \to \infty} J(x)<1$, or
        \item $1 \lesssim v(x)$ and $\dlimsup_{x \to \infty} J(x)/x<0$.
    \end{enumerate} 
\end{thm}
\begin{proof}
In this proof, under each of the assumptions $(a)$, $(b)$, and $(c)$, we construct a positive function $f$ satisfying the conditions of Proposition \ref{prop: expo}, respectively.

For $(a)$, let $f(x) = (\log x)^{\epsilon/2}\mathbbm{1}_{\{x>1\}}+ \mathbbm{1}_{\{x\le 1\}}$, where $\epsilon>0$ is chosen such that $\limsup_{x \to \infty} H_1(x) \le 1-\epsilon$. Applying \eqref{eq: log x f 2} and using $v(x) \gtrsim x^c$ with $c>2$, we obtain $\A f(x) \lesssim -f(x)$, which ensures exponential ergodicity by Proposition \ref{prop: expo}.

For $(b)$, let $f(x)=x^{\epsilon/2} \mathbbm 1_{\{x>1\}} + \mathbbm 1_{\{x\le 1\}}$with $\epsilon>0$ such that $\dlimsup_{x \to \infty} J(x) \le 1-\epsilon$. By \eqref{eq: x f 2} with $p_i$ replaced by $\epsilon/2$ and $v(x) \sim x^2$, we have $\A f(x) \lesssim -f(x)$, which again implies exponential ergodicity.

For $(c)$, define $f(x) = e^{(\epsilon/2)x}$with sufficiently small $\epsilon>0$. By 2 of Proposition \ref{prop: generators}, as $x \to \infty$,
        \begin{align*}
            \mathcal A f(x) &= (\epsilon/2) e^{(\epsilon/2) x} \{m(x)  + (\epsilon/2) v(x) + O(\epsilon^2)O(v(x))\} \notag\\ &= (\epsilon/2) e^{(\epsilon/2) x}v(x)\left\{\frac{J(x)}{x}+\epsilon - (\epsilon/2) + O\left(\epsilon^2\right)\right\}. 
        \end{align*}
       Since $1 \lesssim v(x)$ and $\dlimsup_{x \to \infty} J(x)/x <0$, we can choose $\epsilon$ so that $\A f(x) \lesssim -f(x)$, which completes the proof.
\end{proof}

Next, we provide sufficient conditions for non-exponential ergodicity. Recall that \cite[Proposition 12]{kingman1964} and \cite[Theorem 2]{tweedie1981criteria} give a necessary and sufficient condition: $\bar X$ is exponentially ergodic if and only if there exists $\rho>0$ such that $\rho<q_x:=\sum_{y \neq x} q_{x,y}$ for all $x \in \S$ and $E_x[e^{\rho \tau_x}]<\infty$ for all $x \in \bar \S$, where $\tau_x$ is the first return time to $x$. Therefore, showing that $E_x[\tau_x^\ell] = \infty$ for some $\ell >0$ and $x \in \bar \S$ implies non-exponential ergodicity.

An existing study provided criteria for an irreducible Markov chain $\bar X$ on $\bar \S$ to have an infinite $\ell$th moment for some $\ell>0$. 
\begin{prop}[\cite{menshikov2014explosion}, Theorem 1.5, Corollary 2.13] \label{prop:non-exponential} Let $\bar X$ be an irreducible continuous-time Markov chain on $\bar \S$.
    Let $f\in \text{Dom}_+(\A)$ such that $|\{x: f(x) \le N \}| < \infty$ for every $N>0$. Let $S_a(f)=\{x\in \bar \S : f(x)\le a\}$. Suppose that there exist positive constants $a, c_1,c_2,k$ and $r >1$ such that 
    \begin{enumerate}[(i)]
        \item $\A f(x) \ge -c_1$ for $x \in \bar \S \setminus S_a(f)$,
        \item $f^r \in \text{Dom}_+(\A)$ and $\A f^r(x) \le c_2 f^{r-1}(x)$ for $x \in \bar \S \setminus S_a(f)$, and 
        \item $f^k \in \text{Dom}_+(\A)$ and $\A^k f(x) \ge 0$ for $x \in \bar \S \setminus S_a(f)$.
    \end{enumerate}
    Then we have that $E_x[\tau^\ell_{\bar \S_a(f)}] = \infty$ for all $x \in \bar \S \setminus S_a(f)$ if $\ell\ge k$, where $\tau_{S_a(f)}$ is the first return time to $S_a(f)$. 
\end{prop}

Note that $E_x[\tau^\ell_{\bar \S_a(f)}] = \infty$ implies $E_x[\tau_x^{\ell}] = \infty$ for each $\ell > 0$ and for some $x \in S_a(f)$. This is because there exists $x\in S_a(f)$ such that $P_x(\bar X_{T_1} = y)>0$ for some $y\in \bar \S\setminus S_a(f)$. Then we have $E_x[\tau^{\ell}_x] \ge E_y[\tau^\ell_x]P_x[\bar X_{J_1}=y] = \infty$. Using this fact together with Proposition~\ref{prop:non-exponential}, we now obtain criteria for non-exponential ergodicity.

\begin{thm}[Non-exponential ergodicity]
    Assume that $X$ is positive recurrent.
$X$ is non-exponentially ergodic if either 
        \begin{enumerate}
            \item [(a)] $v(x) \lesssim x^c$ for some $c<2$ and $-\infty  < \liminf_{x \to \infty}J(x) \le \limsup_{x \to \infty}J(x) < \infty$, or
            \item [(b)] $v(x) \lesssim x^c$ for some $c <0$.
        \end{enumerate}    
\end{thm}
\begin{proof}
In this proof, under each of the assumptions $(a)$ and $(b)$, we construct a positive function $f$ satisfying the conditions of Proposition~\ref{prop:non-exponential}.

Assume first that $(a)$ holds.  
By Proposition~\ref{prop:non-exponential}, it suffices to find $f$ and constants $r > 1$ and $k > 0$ such that $\dlim_{x \to \infty} f(x) = \infty$, $\A f(x) \gtrsim -1$, $\A f^r(x) \lesssim f^{r-1}(x)$, and $\A f^k(x) \ge 0$ for large $x$, since $\bar \S_a(f)$ is finite for any constant $a$.  
Define $f(x) = x^{2-c}\mathbbm{1}_{\{x > 1\}} + \mathbbm{1}_{\{x \le 1\}}$ on $\S$. Then $\lim_{x \to \infty} f(x) = \infty$.  
Substituting $p_i = 2-c$ in \eqref{eq: x f 2} yields $\A f(x) \gtrsim -1$.  
Fix an arbitrary constant $r > 1$.  
Using the assumption $v(x) \lesssim x^c$ and replacing $p_i$ with $(2-c)r$ in \eqref{eq: x f 2}, we obtain
        \begin{align*}
            \A f^r(x) \lesssim \frac{x^{(2-c)r}v(x)}{x^2} \lesssim x^{(2-c)(r-1)}\sim f(x)^{r-1}.
        \end{align*}
        Finally, take $k > 0$ so that $\liminf_{x \to \infty} \{J(x) + (2-c)k - 1\} > 0$.  
Then, by substituting $(2-c)k$ for $p_i$ in \eqref{eq: x f 2}, we have $\A f^k(x) \ge 0$ for large $x$.

Next, assume that $(b)$ holds.  
Since $v(x) = \frac{1}{2}\sum_{\eta \in \Gamma} \eta^2 \lambda_\eta(x) \lesssim x^c$ for some $c < 0$, the total transition rate $q_x = \sum_{\eta \in \Gamma}\lambda_\eta(x)$ converges to $0$ as $x \to \infty$.  
Then there exists no $\rho > 0$ such that $\rho < q_x$ for all $x \in \S$.

\end{proof}

\section{Criteria for transition rates with Laurent-type asymptotics}\label{sec : main-ctmc}

Under Assumption~\ref{basic assumption}, we further restrict our attention to continuous-time Markov chains $X$ whose transition rates admit a \textit{Laurent-type asymptotic expansion}, which allows us to simplify the criteria for their dynamical properties. 

\begin{defn}
We say that the collection of transition rates $\{\lambda_\eta(x)\}_{\eta \in \Gamma}$ of $X$ is defined with a \emph{Laurent-type asymptotic} if there exists an integer $R \in \Z$ such that 
\begin{align}
\begin{split}\label{eq:laurent}
\lambda_\eta(x) = a_\eta x^R + b_\eta x^{R-1} + O(x^{R-2}) \quad \text{for each \(\eta \in \Gamma\)}, 
\end{split}
\end{align}
where \(a_\eta,b_\eta \in \R\) and \(a_{\eta} > 0\) for at least one \(\eta \in \Gamma\).
Here, $R$ denotes the maximal degree among all transition rates. Note that $a_\eta$ or $b_\eta$ may be zero for some $\eta \in \Gamma$.
\end{defn}

Under this assumption, both $m(x)$ and $v(x)$ admit Laurent-type asymptotic forms whose degrees and coefficients are directly determined as
\begin{align}\label{eq:m and v under laurent}
    m(x) = \alpha x^{R} + \gamma x^{R-1} + O(x^{R-2}), 
    \qquad
    v(x) = \vartheta x^R + O(x^{R-1}),
\end{align}
where
\begin{align}\label{eq : def-ctmc}
  \alpha = \sum_{\eta \in \Gamma} \eta a_\eta, 
  \quad 
  \gamma = \sum_{\eta \in \Gamma} \eta b_\eta,
  \quad \text{and} \quad 
  \vartheta = \frac{1}{2}\sum_{\eta \in \Gamma} \eta^2 a_\eta.
\end{align}

\begin{rem}
If all transition rates are rational functions of $x$, then they satisfy the Laurent-type asymptotic condition.
\end{rem}

In the previous section, we derived sufficient conditions for various dynamical properties in terms of $m(x)$ and $v(x)$. These conditions, however, are not if-and-only-if, and hence do not yield a complete classification. In particular, there exist critical regimes where neither recurrence nor transience can be determined directly. For instance, if
\[
\limsup_{x \to \infty} H_1(x) \ge 1 \ge \liminf_{x \to \infty} H_1(x),
\]
then neither Theorem~\ref{thm: recu} nor Theorem~\ref{thm: tran} applies.  
Under the Laurent-type asymptotic assumption~\eqref{eq:laurent}, however, such ambiguity does not occur. Indeed, we have as $x \to \infty$,
\begin{align*}
    H_1(x) 
    &= \left(1 + O\left(\frac{1}{x}\right)\right)
    \frac{\log x \, (\alpha x^{R+1} + (\gamma-\vartheta)x^R)}{\vartheta x^R}
    + O\left(\frac{\log x}{x}\right).
\end{align*}
Thus,
\begin{align}\label{eq: classification}
    \lim_{x \to \infty} H_1(x)
    = \frac{\log x (\alpha x + \gamma - \vartheta)}{\vartheta}
    =
    \begin{cases}
       -\infty, & \text{if }\alpha \le 0 \text{ and }\min(\alpha, \gamma-\vartheta)<0, \\[3pt]
       0, & \text{if }\alpha = \gamma-\vartheta = 0, \\[3pt]
       -\infty, & \text{if }\alpha \ge 0 \text{ and }\max(\alpha, \gamma-\vartheta)>0,
    \end{cases}
\end{align}
which shows that recurrence and transience can be completely characterized in terms of the parameters $\alpha$, $\gamma$, and $\vartheta$.

In fact, under~\eqref{eq:laurent}, all dynamical properties of $X$ can be classified solely by the parameters $(\alpha,\gamma,\vartheta,R)$. More precisely, for each $(\alpha,\gamma,\vartheta,R) \in \R \times \R \times \R_{>0} \times \Z$, one can determine whether $X$ is explosive or non-explosive, whether it is transient, null recurrent, or positive recurrent, and whether it is exponentially ergodic.

In this section, we focus on the cases $R \le 0$. The case $R \ge 1$ was largely analyzed in~\cite{XuHansenWiuf2022}. While some parameter regimes—particularly regarding exponential ergodicity—were not fully classified there, we address those in Section~\ref{sec: every single}.

\begin{thm}\label{thm: main for laurante}
Let $R \le 0$. Then $X$ is non-explosive and
\begin{enumerate}
    \item[(i)] transient if $\alpha>0$, or if $\alpha=0$ and $\gamma>\vartheta$,
    \item[(ii)] null recurrent if $\alpha=0$ and $(R-1)\vartheta\le\gamma\le \vartheta$, 
    \item[(iii)] non-exponentially ergodic if $\alpha=0$ and $\gamma<(R-1)\vartheta$, or if $\alpha <0$ and $R<0$, and 
    \item[(iv)] exponentially ergodic if $\alpha<0$ and $R=0$.
\end{enumerate}
\end{thm}

\begin{proof}
Note that
\[
m(x) = x^R(\alpha + \gamma x^{-1} + O(x^{-2})), 
\qquad 
v(x) = x^R(\vartheta + O(x^{-1})),
\]
and hence
\[
m(x)x - v(x) = x^R(\alpha x + \gamma - \vartheta + O(x^{-1})).
\]
Since $R \le 0$, non-explosivity of $X$ follows from part~$2(a)$ of Theorem~\ref{thm:main}.  
Throughout this proof, we set $c = R$.

From~\eqref{eq: classification}, case~$(i)$ and the recurrence in $(ii)$--$(iv)$ follow directly from Theorem~\ref{thm:main}.  

Now assume $(ii)$.  
If $(R-1)\vartheta < \gamma$, then $\dlim_{x \to \infty} J(x) = \gamma/\vartheta > c-1$.  
If $(R-1)\vartheta = \gamma$, then
\[
\dlim_{x \to \infty} H_{c-1}(x) 
= \dlim_{x \to \infty} \frac{\log x (\gamma-(R-1)\vartheta)}{\vartheta} 
= 0.
\]
In both cases, null recurrence follows from parts~$6(a)$ and~$6(b)$ of Theorem~\ref{thm:main}, respectively.

Next, observe that $\dlim_{x \to \infty} J(x) = (\alpha x + \gamma)/\vartheta < c-1$ under the assumptions of either $(iii)$ or $(iv)$, implying positive recurrence by part~$5(a)$ of Theorem~\ref{thm:main}.  
In case $(iii)$, the first condition yields $\dlim_{x \to \infty} J(x) = \gamma/\vartheta$, which ensures non-exponential ergodicity by part~$8(a)$ of Theorem~\ref{thm:main}.  
The second condition similarly implies non-exponential ergodicity by part~$8(b)$ of Theorem~\ref{thm:main}.

Finally, for case $(iv)$, since $R=0$ implies $1 \lesssim v(x)$ and $\alpha<0$ gives $\lim_{x \to \infty} J(x)/x$ $= \alpha < 0$, exponential ergodicity follows from part~$7(c)$ of Theorem~\ref{thm:main}.
\end{proof}

\section{Applications}\label{sec : applications}

The main application of Theorem~\ref{thm: main for laurante} lies in \emph{stochastic reaction networks}. Reaction networks provide a graphical representation of biochemical interaction systems, such as
\begin{align*}
    A+B \xrightarrow{\kappa_1} 2A \xrightarrow{\kappa_2} \emptyset, 
    \qquad 
    B\xrightarrow{\kappa_3} A.
\end{align*}

In general, a reaction network consists of reactions of the form
\begin{align*}
    \nu_{i1}S_1 + \nu_{i2}S_2 + \dots + \nu_{id}S_d 
    \xrightarrow{\kappa_i} 
    \nu'_{i1}S_1 + \nu'_{i2}S_2 + \dots + \nu'_{id}S_d, 
    \qquad i = 1,\dots,k.
\end{align*}
We denote such a network by $(\Sp,\C,\Re)$, where  
$\Sp=\{S_1,\dots,S_d\}$ is the set of \emph{species},  
$\C=\{\sum_j \nu_{ij}S_j, \sum_j \nu'_{ij}S_j : i=1,\dots,k\}$ is the set of \emph{complexes}, and  
$\Re=\{\sum_j \nu_{ij}S_j \to \sum_j \nu'_{ij}S_j : i=1,\dots,k\}$ is the set of \emph{reactions}.  
The associated collection of \emph{rate constants} is $\K=\{\kappa_i : i=1,\dots,k\}$.  
Then $(\Sp,\C,\Re,\K)$ specifies a \emph{reaction system}, which determines the corresponding stochastic dynamics.

The generator of the associated continuous-time Markov chain $\bar X$ on $\mathbb Z_{\ge 0}^d$ is given by
\begin{align*}
    \mathcal A f(x)
    = \sum_{i=1}^R \lambda_i(x)\big(f(x+\eta_i) - f(x)\big),
    \qquad 
    \text{for all } f \in \text{Dom}_+(\mathcal A) \text{ and } x \in \mathbb Z_{\ge 0}^d,
\end{align*}
where $\eta_i  \in \mathbb Z^d$ is the \emph{reaction vector} with components $\eta_{ij} = \nu'_{ij} - \nu_{ij}$, representing the net change in the $i$th reaction.  
As indicated by the generator, the transition of $\bar X$ is given by $\eta_i$, and its corresponding rate is $\lambda_i$.

Under the \emph{mass-action kinetics}, the transition rate $\lambda_i(x)$ is defined by
\begin{align}\label{eq:mass}
    \lambda_i(x) 
    = \kappa_i \prod_{j=1}^d \frac{x_j!}{(x_j - \nu_{ij})!} 
    \mathbbm 1_{\{x_j \ge \nu_{ij}\}}.
\end{align}
For example, for the reaction $2S_1 + S_2 \xrightarrow{\kappa} S_3$, the corresponding transition is $\eta=(-2,-1,1)$, and the mass-action transition rate is $\lambda(x)=\kappa x_1(x_1-1)x_2$.  
However, as previously noted, when a multi-species reaction network under mass-action is reduced to an effective one-dimensional model, the resulting transition rates can often be approximated by more general forms such as rational functions (see Section~\ref{sec : main-ctmc}).

\subsection{Every ergodic single-species reaction network under mass-action is exponentially ergodic}\label{sec: every single}

In~\cite{kim2025path}, a special class of continuous-time Markov chains associated with reaction networks was introduced, which admits non-exponential ergodicity.  
In particular, when the chain reaches a certain part of the boundary—correspon- ding to the extinction of some species—it may become trapped in a long-lasting loop, leading to slow mixing.  
For example, consider the continuous-time Markov chain $X$ associated with the following mass-action reaction network:
\begin{align*}
A+B \rightleftharpoons \emptyset, 
\qquad 
B \rightleftharpoons 2B.
\end{align*}
When the associated Markov chain $\bar X$ reaches a state of the form $(n,0)$ for large $n$, it is likely to oscillate between $(n,0)$ and $(n+1,1)$ for an extended period with high probability (see Example~2.1 in~\cite{kim2025path}).  
Since such boundary behavior arises from the extinction of one species, one can expect that this slow-mixing phenomenon does not occur in single-species reaction networks.

For a single-species reaction network under mass-action, one might still wonder whether non-exponential ergodicity could appear without such boundary effects, namely under either condition $8(a)$ or $8(b)$ of Theorem~\ref{thm:main}.  
We will show that under mass-action kinetics, these conditions can never be satisfied.

Exponential ergodicity of ergodic single-species reaction networks  under mass-action was proved in~\cite{XuHansenWiuf2022} except for the case of $R=2$ and $\gamma<\vartheta$. For this remaining case, we prove exponential ergodicity.
\begin{prop}\label{prop:R2}
Let $X$ be a one-dimensional continuous-time Markov chain on $\mathbb Z_{\ge 0}$.  
If $R=2$, $\alpha=0$, and $\gamma<\vartheta$, then $X$ is exponentially ergodic.
\end{prop}

\begin{proof}
We will show that when $R=2$, the conditions of part~$7(b)$ in Theorem~\ref{thm:main} hold under mass-action kinetics.  
Since $R=2$, we have $v(x) \sim x^2$.  
Furthermore, by~\eqref{eq:m and v under laurent}, $\alpha=0$ and $\gamma<\vartheta$ imply
\[
\lim_{x\to \infty} J(x)
= \lim_{x\to \infty} \frac{\gamma x^2}{\vartheta x^2}
= \frac{\gamma}{\vartheta}
< 1,
\]
which establishes the claim.
\end{proof}

Based on Proposition~\ref{prop:R2} and the preceding discussion, we obtain the following complete classification for single-species reaction systems under mass-action.

\begin{cor}
Let $X$ be the continuous-time Markov chain associated with a reaction system $(\Sp,\C,\Re,\K)$ under mass-action kinetics~\eqref{eq:mass}.  
If $X$ is ergodic, then it is exponentially ergodic.
\end{cor}

\subsection{One-dimensional Michaelis-Menten kinetics}

In real biochemical systems, reaction rates do not necessarily follow mass-action kinetics. A well-known and biologically relevant example is the Michaelis-Menten kinetics, which describes the rate of enzymatic reactions where a substrate is converted into a product with the help of an enzyme~\cite{segel1975enzyme}. In the one-dimensional case, the Michaelis-Menten mechanism can be written as
\begin{align}\label{eq:MM}
E + nA \rightleftharpoons C \rightarrow E + P,
\end{align}
where $E$ is an enzyme, $n,m$ are positive integers, and $C$ is the enzyme–substrate complex.
If $E + nA \rightleftharpoons C$ represents the formation of the complex $C$ that actually occurs through a sequence of reversible binding reactions, then the rate of the reaction follows a rational function. 
Specifically, when we denote by $EA_i$ the intermediate complex formed by one enzyme $E$ and $i$ molecules of $A$, the elementary reversible steps are $EA_i + A \rightleftharpoons EA_{i+1}$, where $EA_n = C$~\cite{segel1975enzyme, marangoni2003enzyme}.
Assuming that the dissociation constants of these intermediate reactions are equal to $K$, the quasi steady-state approximation reduces \eqref{eq:MM} to $nA\to P$  whose the reaction intensity takes the form
\begin{align}\label{eq:mm-rate}
\lambda_{c}(x; n, V, K)
= \frac{V x^{n}}{1 + (x/K) + \cdots + (x/K)^n}
= V K^n - V K^{n+1} x^{-1} + O(x^{-2}),
\end{align}
 where $x$ is the copy number of $A$, and $V$ is the maximal rate ~\cite{segel1975enzyme, johnson2011original, briggs1925note}.

As a simple application of such special types of  Michaelis-Menten kinetics, we can consider a birth–and-death process where each reaction rate follows~\eqref{eq:mm-rate}:
\begin{align}\label{eq : birth-death-mm}
\begin{split}
n_1 A \xrightarrow{\lambda_{-c_1}(x; n_1, V_1, K_1)} (n_1 - c_1)A, \quad
n_2 A \xrightarrow{\lambda_{c_2}(x; n_2, V_2, K_2)} (n_2 + c_2)A.
\end{split}
\end{align}
Then, by~\eqref{eq : def-ctmc}, we have
$R = 0,~
\alpha = -c_1V_1K_1^{n_1} + c_2V_2 K_2^{n_2}, ~
\gamma = c_1 V_1 K_1^{n_1 + 1} - c_2 V_2 K_2^{n_2 + 1}, ~
\theta = \tfrac{1}{2}(c_1^2 V_1 K_1^{n_1}+ c_2^2 V_2 K_2^{n_2})$.
Thus, by Theorem~\ref{thm: main for laurante}, the stochastic reaction network~\eqref{eq : birth-death-mm} is
\begin{enumerate}[(a)]
\item exponentially ergodic if $c_1V_1K_1^{n_1} > c_2V_2 K_2^{n_2}$,
\item transient if $c_1V_1K_1^{n_1} <c_2V_2 K_2^{n_2}$, or if $c_1V_1K_1^{n_1} = c_2V_2 K_2^{n_2}$ and $K_1 - K_2 > \frac{c_1 + c_2}{2}$,
\item non-exponentially ergodic if $c_1V_1K_1^{n_1} = c_2V_2 K_2^{n_2}$ and $K_1 - K_2 < -\tfrac{c_1 + c_2}{2}$,
\item null recurrent if $c_1V_1K_1^{n_1} = c_2V_2 K_2^{n_2}$ and $-\frac{c_1 + c_2}{2} \leq K_1 - K_2 \leq \tfrac{c_1 + c_2}{2}$.
\end{enumerate}

Note that if $c_1 = c_2 = 1$, then~\eqref{eq : birth-death-mm} reduces to a simple birth–death process.
In this case, the same classification can also be derived using the classical results of~\cite{mao2004ergodic, karlin1957classification, norris1997jr}.
However, Theorem~\ref{thm: main for laurante} also provides the classification for birth–death processes with multiple jump sizes.

As examples, when $(V_1, K_1, V_2, K_2) = (1,1,1,1)$ and $(n_1, c_1, n_2, c_2) = (3,2,2,1)$ in~\eqref{eq : birth-death-mm}, the associated Markov chain is exponentially ergodic.
In contrast, when $(V_1, K_1, V_2, K_2) = (3,1,2,3)$ and $(n_1, c_1, n_2, c_2) = (4,2,1,1)$, the associated Markov chain is non-exponentially ergodic. We display $\log{( \Vert P(X(t)\in \cdot)-\pi(\cdot)\Vert_{tv} )}$ for these models in Figure \ref{fig:exp-erg-ex}, where $\pi$ is the stationary distribution.  

\begin{figure}[h!]

\centering
    \includegraphics[width=0.85\linewidth]{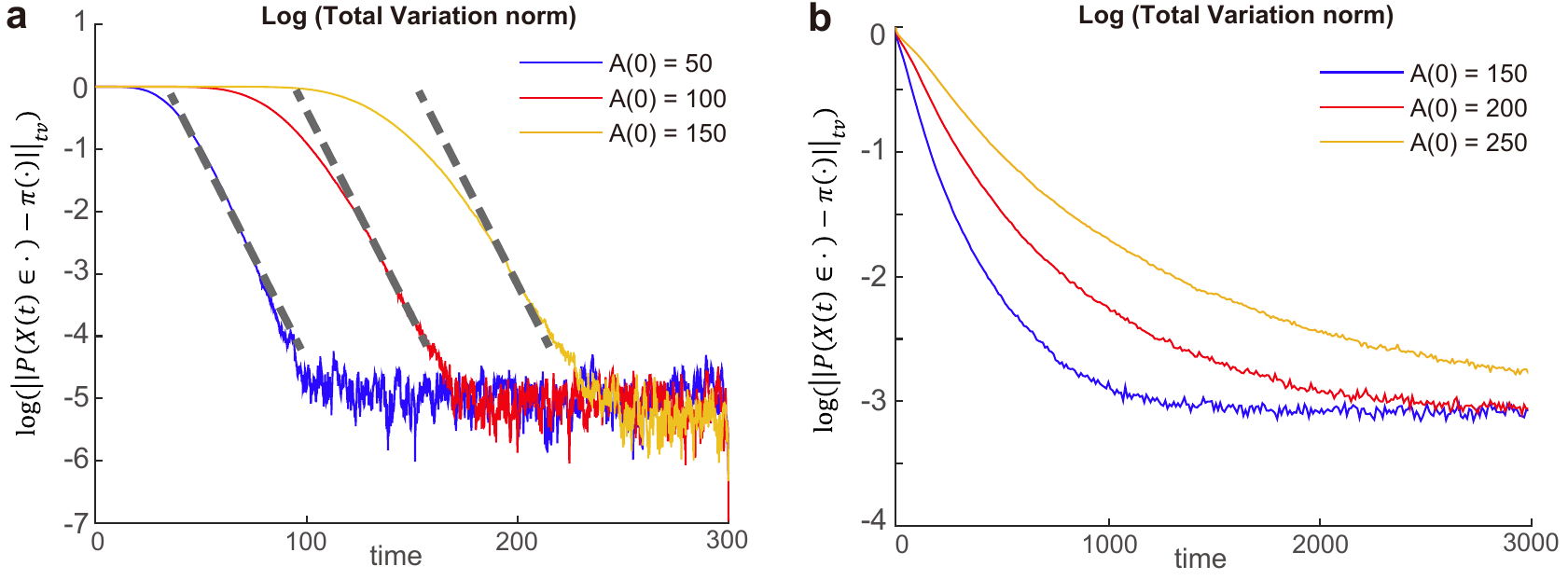}
    \caption{\textbf{a}, \textbf{b}: $\log(||P(X(t)\in \cdot) - \pi(\cdot)||_{tv})$ for the cases $(V_1,K_1,V_2,K_2)=(1,1,1,1)$ and $(n_1,c_1,n_2,c_2)=(3,2,2,1)$ (a), and $(V_1,K_1,V_2,K_2)=(3,1,2,3)$ and $(n_1,c_1,n_2,c_2)=(4,2,1,1)$ (b) of~\eqref{eq : birth-death-mm} with different initial values $A(0)$. While $\log(||P(X(t)\in \cdot) - \pi(\cdot)||_{tv})$ decays with almost identical slopes in the left plot across initial values, the decay rates differ in the right plot, indicating exponential and non-exponential ergodicity, respectively.
Each curve was obtained from 50,000 simulated trajectories using the Gillespie algorithm~\cite{Gillespie77}. The plateau in each plot emerges due to Monte-Carlo errors.}
\label{fig:exp-erg-ex}
\end{figure}

\subsection{Haldane equation}

Another example of non–mass-action kinetics is the Haldane equation, which describes substrate inhibition kinetics~\cite{sonnad2004solution, haldane1930enzymes}.
The Haldane mechanism introduces an additional inactive complex $EA_{n+1}=C+A$ to the Michaelis–Menten framework \eqref{eq:MM}.
 In this case, the quasi steady-state approximation, the trantion rate of the reduced model $nA \to P$ takes the form
\begin{align}\label{eq : haldane}
\lambda_c(x; n, V, K)
= \frac{V x^n}{1 + (x/K) + \cdots + (x/K)^{n+1}}
= V K^{n+1} x^{-1} - V K^{n+2} x^{-2} + O(x^{-3}),
\end{align}
where $x$ is the copy number of $A$.

As an analogy with \eqref{eq : birth-death-mm}, we consider a birth–and-death process where each reaction rate follows~\eqref{eq : haldane}:
\begin{align}\label{eq : birth-death-haldane}
n_1 A \xrightarrow{\lambda_{-c_1}(x; n_1, V_1, K_1)} (n_1 - c_1)A, \quad
n_2 A \xrightarrow{\lambda_{c_2}(x; n_2, V_2, K_2)} (n_2 + c_2)A.
\end{align}

Then, by~\eqref{eq : def-ctmc}, we obtain
$
R = -1, ~\alpha = -c_1 V_1 K_1^{n_1 + 1} + c_2 V_2 K_2^{n_2 + 1}, ~\gamma = c_1 V_1 K_1^{n_1 + 2} - c_2 V_2 K_2^{n_2 + 2}, ~
\theta = \tfrac{1}{2}(c_1^2 V_1 K_1^{n_1 + 1} + c_2^2 V_2 K_2^{n_2 + 1})$.
Hence, by Theorem~\ref{thm: main for laurante}, the stochastic reaction network~\eqref{eq : birth-death-haldane} satisfies
\begin{enumerate}[(a)]
\item transient if $c_1 V_1 K_1^{n_1 + 1} < c_2 V_2 K_2^{n_2 + 1}$, or if $c_1 V_1 K_1^{n_1 + 1} = c_2 V_2 K_2^{n_2 + 1}$ and $K_1 - K_2 > \tfrac{c_1 + c_2}{2}$
\item null recurrent if $c_1 V_1 K_1^{n_1 + 1} = c_2 V_2 K_2^{n_2 + 1}$ and $-(c_1 + c_2) \leq K_1 - K_2 \leq \tfrac{c_1 + c_2}{2}$,
\item non-exponentially ergodic if $c_1 V_1 K_1^{n_1 + 1} > c_2 V_2 K_2^{n_2 + 1}$, or if $c_1 V_1 K_1^{n_1 + 1} =$

$c_2 V_2 K_2^{n_2 + 1}$ and $K_1 - K_2 < -(c_1 + c_2)$.
\end{enumerate}

As examples, when $(V_1, K_1, V_2, K_2) = (16,1,1,4)$ and $(n_1, c_1, n_2, c_2) = (1,1,1,1)$ in \eqref{eq : birth-death-haldane}, the associated Markov chain is non-exponentially ergodic.
In contrast, when $(V_1, K_1, V_2, K_2) = (1,2,4,1)$ and $(n_1, c_1, n_2, c_2) = (1,1,1,1)$ in \eqref{eq : birth-death-haldane}, the associated Markov chain is null recurrent.

\subsection{General rational kinetics}\label{sec:reduction}

As described in the Introduction, a reaction network with multiple species under mass-action can be approximated by a single-species reaction network under rational kinetics.  

For example, consider a mass-action reaction system $(\Sp,\C,\Re,\K)$ given by
\begin{align}
&\emptyset \xrightarrow{1} 2A, 
\quad 
B+2A\xrightarrow{2/V} B+A, 
\quad 
B+A\xrightarrow{6/V} B, \label{ex:reduction row 1}\\[2mm]
&\emptyset \xrightarrow{VU}B, 
\quad 
B+2A \xrightarrow{U} 2A, 
\quad 
B\xrightarrow{U} \emptyset, 
\label{ex:reduction row 2}
\end{align}
where the scaling parameters $U$ and $V$ represent the acceleration rate and the system volume, respectively.  
Let $X_A(0)=1$ and $X_B(0)=V$. Under this initial condition, the reaction intensities of the three reactions in \eqref{ex:reduction row 1} and those in \eqref{ex:reduction row 2} are of order $1$ and $V$, respectively. Thus the system admits a natural time-scale separation.  

Remarkably, the scaled process $X_B(t)/V$ converges in probability to $x_B(t)$, where $x_B(t)$ solves
\begin{align*}
    x_B'(t)=U(1-(X_A(X_A-1)+1)x_B), \quad x_B(0)=1.
\end{align*}
This convergence follows from the relationship between deterministic and stochastic models of reaction systems \cite{kurtz1970solutions, AndKurtz2011}, which is based on the law of large numbers applied to the random time change representation:
\begin{align*}
    \frac{X_B(t)}{V}
    &=\frac{X_B(0)}{V}
    +\frac{N_1}{V}(VUt)
    -\frac{N_2}{V} \!\left(VU \int_0^t X_A(X_A-1)\frac{X_B(s)}{V}ds \right)
    -\frac{N_3}{V} \!\left(U \int_0^t \frac{X_B(s)}{V}ds \right),
\end{align*}
where $N_i(t)$'s are independent unit-rate Poisson processes.  

When $U$ is sufficiently large, $x_B(t)$ rapidly converges to the steady state $\frac{1}{X_A(X_A-1)+1}$ (Figure~\ref{fig:approx}c).  
During this fast convergence, the copy number $X_A(t)$ remains nearly constant, as the reactions in \eqref{ex:reduction row 1}, which only affect $X_A$, occur on a much slower time scale.  
Once $x_B$ has reached its quasi-steady value $\frac{1}{x_A(x_A-1)+1}$ for $X_A(t)=x_A$, a subsequent jump of $X_A$ to $x'_A$ induces a quick relaxation of $x_B$ to another quasi-steady value $\frac{1}{x'_A(x'_A-1)+1}$ (Figure~\ref{fig:approx}c).  
Although this behavior can be rigorously analyzed via singular perturbation techniques, we describe it here only heuristically.

Because $X_B/V$ converges almost instantaneously, the stochastic process associated with \eqref{ex:reduction row 1}–\eqref{ex:reduction row 2} can be reduced to a one-dimensional model
\begin{align}\label{ex:reduced}
A\to \emptyset \to 2A,    
\end{align}
whose transition rates are given by
\begin{align*}
\lambda_1(x)=\frac{2x(x-1)+6x}{x(x-1)+1}=2+6x^{-1}+O(x^2),
\qquad 
\lambda_2(x)=1,
\end{align*}
(see Figure~\ref{fig:approx}a,b).  
Hence, in the reduced model \eqref{ex:reduced}, we have $R=0$, $\alpha=0$, $\gamma=-6$, and $\vartheta=4$.  
By Theorem~\ref{thm: main for laurante}, $X$ for \eqref{ex:reduced} is therefore non-exponentially ergodic.  

\begin{figure}[!h]
    \centering
    \includegraphics[width=0.9\linewidth]{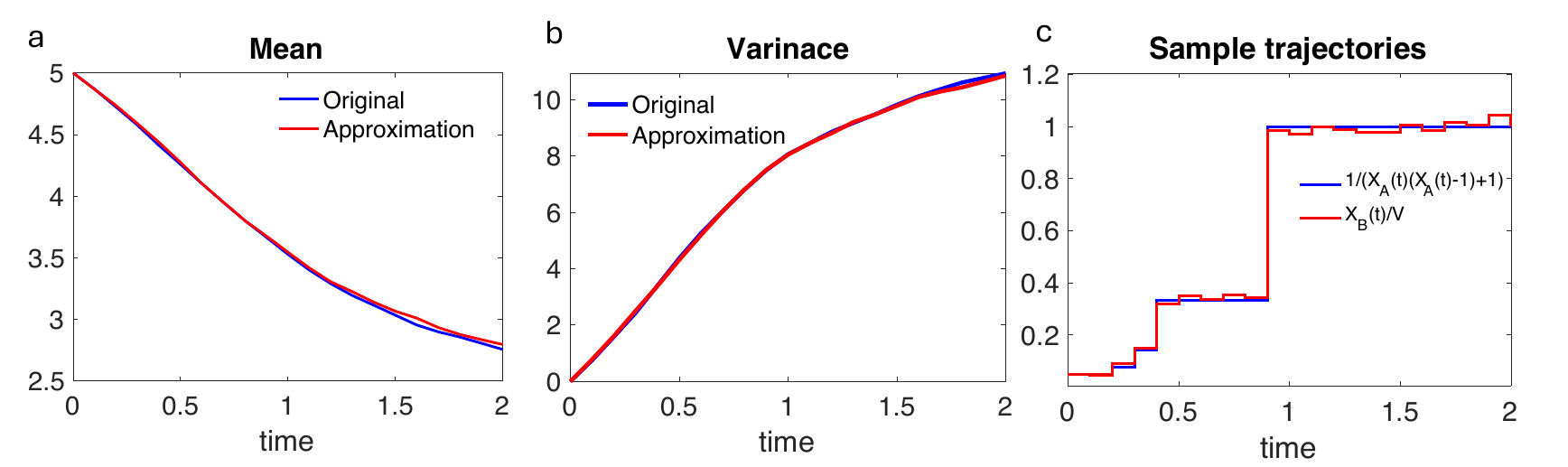}
    \caption{\textbf{a}, \textbf{b}: The mean (a) and variance (b) of $X$ in the original model \eqref{ex:reduction row 1}–\eqref{ex:reduction row 2} and in the approximation \eqref{ex:reduced}, with $U=10^2$ and $V=10^3$ based on $10^4$ trajectories. 
    \textbf{c}: Sample trajectories of $1/(X_A(t)(X_A(t)-1)+1)$ and $X_B(t)/V$ in \eqref{ex:reduction row 1}–\eqref{ex:reduction row 2}, showing that $X_B(t)/V$ rapidly converges to $1/(X_A(X_A-1)+1)$ after each jump of $X_A$.}
    \label{fig:approx}
\end{figure}

\medskip

In general, for a two-species reaction network, suppose that the reactions are separated into \emph{fast} and \emph{slow} reactions such that  
\begin{enumerate}[(a)]
    \item $Y$ evolves only through fast reactions catalyzed by $X$, and  
    \item $X$ evolves only through slow reactions catalyzed by $Y$.
\end{enumerate}
Here we say that a reaction is \emph{catalyzed by $X$} if the reactant involves $X$ but the reaction does not change the copy number of $X$.  
For example, $X+2Y\to X+Y$ is catalyzed by $X$. 
Then $X$ can be approximately modeled by a reduced system whose transition rates are arbitrary Laurent-type functions, including rational functions.

To formalize this reduction, it is convenient to use the notation 
\[
x^{(n)}:=\frac{x!}{(x-n)!}\mathbbm 1_{x\ge n},
\]
for $x\in \mathbb Z_{\ge 0}$, to express the mass-action form compactly.  
Since general rational transition rates arise from such reductions, we represent rational functions in terms of $x^{(n)}$ rather than $x^n$.

\begin{prop}\label{prop:general rational function}
Let $L,K\in\mathbb Z_{\ge 0}$, and define $f_1(x)=\sum_{i=0}^K s_i x^{(i)}$ and $f_2(x)=\sum_{j=0}^L r_j x^{(j)}$, where $f_1$ is nonnegative and $f_2$ is strictly positive for all $x\in\mathbb Z_{\ge 0}$.  
Let $g(x)=f_1(x)/f_2(x)$.  
Then there exists a reaction system $(\Sp,\C,\Re,\K)$ satisfying:
\begin{enumerate}
    \item $\Sp=\{X,Y\}$;
    \item every reaction in $\Re$ is catalyzed by $X$;
    \item each rate constant $\kappa_i\in\K$ depends on parameters $V$ and $U$;
    \item for the associated Markov chain $(X,Y)$ with $Y(0)=O(V)$ and arbitrary $X(0)=x$, if $U$ is sufficiently large, then
    \begin{align*}
        \dlim_{U\to \infty}\lim_{V\to \infty}
        P\!\left(\sup_{\delta\le s\le t}
        \left|\frac{Y(s)}{V}-g(x)\right|>\epsilon\right)=0,
    \end{align*}
    for any $t>0$, $\delta\in(0,t)$, and $\epsilon>0$, under mass-action kinetics.
\end{enumerate}
\end{prop}

\begin{proof}
Let
\begin{align*}
    \Re
    &=\{iX+Y \xrightarrow{-s_i U} iX: s_i<0\}
    \cup \{iX+Y\xrightarrow{s_i U} iX+2Y: s_i>0\}\\
    &\quad \cup \{jX+2Y \xrightarrow{-r_j U/V} jX+Y: r_j<0\}
    \cup \{ jX+2Y \xrightarrow{r_j U/V} jX+3Y: r_j>0\}.
\end{align*}
These reactions leave $X(t)\equiv X(0)=x$ invariant for all $t$, since none alters the copy number of $X$.  
Hence we only track $Y(t)$, which evolves as a one-dimensional Markov chain with reactions  
\begin{align*}
    &\{Y \xrightarrow{-s_i U x^{(i)}} \emptyset: s_i<0\}
\cup 
\{Y\xrightarrow{s_i U x^{(i)}} 2Y : s_i>0\}
\\ &\cup\{2Y \xrightarrow{-r_j Ux^{(j)}/V} Y: r_j<0\}
\cup
\{ 2Y \xrightarrow{r_j Ux^{(j)}/V} 3Y: r_j>0\},
\end{align*}
under mass-action kinetics.  
The mean transition rate of $Y$ thus has degree $R=2$ with leading coefficient $\alpha=-Uf_2(x)/V<0$, ensuring non-explosivity by \cite{XuHansenWiuf2022}.  

By \cite[Theorem~2.11]{kurtz1970solutions}, for each fixed $t>0$ and any $\epsilon>0$,
\[
\lim_{V\to \infty}P\!\left(\sup_{s\le t}\left|\frac{Y(s)}{V}-y(s)\right|>\epsilon\right)=0,
\]
where $y(t)$ satisfies $y'(t)=Uy(f_1(x)-f_2(x)y)$ with $y(0)=Y(0)/V$.  
As $U\to\infty$, $y(t)$ converges uniformly on $[\delta,t]$ to $g(x)$ for any $\delta\in(0,t)$.  

Therefore, for any $\epsilon>0$, there exists $U_0>0$ such that for $U>U_0$, we have that $
\sup_{\delta\le s\le t}|y(s)-g(x)|<\tfrac{\epsilon}{2}$.
Hence, for all $U>U_0$,
\begin{align*}
   P\!\left(\sup_{\delta\le s\le t}\left|\frac{Y(s)}{V}-g(x)\right|>\epsilon\right)
   &\le 
   P\!\left(\sup_{\delta\le s\le t}\left|\frac{Y(s)}{V}-y(s)\right|>\tfrac{\epsilon}{2}\right).
\end{align*}
Letting $V\to\infty$ yields the desired result.
\end{proof}

Finally, suppose that some slow reactions catalyzed by $Y$ as in \eqref{ex:reduction row 1} are added to those considered in Proposition~\ref{prop:general rational function}.  
Then, when $X(t)=x$, the fast subsystem ensures that $Y(t)/V\approx g(x)$ for each $t$.  
Consequently, $X$ behaves as a one-dimensional Markov chain whose transition rates are given by $g(x)$.  
This provides a systematic framework for studying high-dimensional reaction systems via reduction to one-dimensional Markov chains.

\section{Discussion}
In this paper, we established criteria for various dynamical properties of continuous-time Markov chains on $\S \subseteq \mathbb Z_{\ge 0}$.
These criteria were formulated in terms of the asymptotic behaviors of the mean and variance of the transition rates (Theorem~\ref{thm:main}).
Based on these results, when the transition rates admit a Laurent-type expansion~\eqref{eq:laurent}, we provided a complete classification of the dynamical features of continuous-time Markov chains on $\S$ in terms of the degrees and coefficients of their transition rates (Theorem~\ref{thm: main for laurante}).

As applications of this classification, we showed that every single-species stochastic reaction system under mass-action kinetics is exponentially ergodic whenever it is positive recurrent.
Using the same framework, we further analyzed the dynamical properties of birth-and-death processes with Michaelis–Menten and Haldane kinetics.

Finally, we demonstrated that reaction networks with a single species and rational transition rates admitting Laurent-type expansions can approximate multi-species mass-action reaction systems.
Beyond reproducing the probability distributions of the original reaction systems, such one-dimensional systems can also provide insights into their dynamical behaviors.
For instance, for a multi-dimensional Markov chain $\bar X(t)=(\bar X_1(t),\dots,\bar X_d(t))$, let $\tau_{x_1}=\inf\{t: \bar X_1(t)=x_1\}$.
If $P(\tau_{x_1}=\infty)>0$, then the entire system $\bar X$ is transient.
That is, the projection of $\bar X$ onto a single coordinate can determine the transience of the original system.
Similarly, if there exists $x_1\in \mathbb Z_{\ge 0}$ such that $|P(\bar X_1(t)=x_1)-\pi(x_1)|$ does not decay exponentially in time $t$ for the stationary distribution $\pi$, then $\bar X$ is non-exponentially ergodic. That is, a behavior of a single coordinate of $\bar X$ can be used to determine dynamical features of $\bar X$. 

A similar viewpoint was already proposed in~\cite{ballif2025partition}, where the positive recurrence of high-dimensional reaction networks is studied via their projected Markov chains on one-dimensional state spaces.
In this sense, a one-dimensional projection can serve as a lens through which the behavior of the entire system can be examined.
Since such projections can be approximated by one-dimensional Markov chains, as shown in Section~\ref{sec:reduction}, our classification provides a simple and effective methodology for analyzing the dynamical properties not only of one-dimensional Markov chains but also of higher-dimensional ones.

\appendix

\section{Proofs of Proposition \ref{prop: generators} and Corollory \ref{cor: for the sup assumption}}\label{sec : generator-expansion}

\textit{Proof of Corollary \ref{prop: generators}.} 
\noindent 1. 
    Since $x^p (\log x)^q$ is analytic on $[1,\infty)$, by Taylor expansion, for each $\eta \in \Gamma$ and $x > 1 + |\eta|$, we have
    $
        f(x+\eta) = f(x) + \eta f'(x) + \frac{1}{2}\eta^2 f''(x) + \frac{1}{6}f'''(x+c),
    $
    where $c$ depends on $x$ and lies between $x$ and $x+\eta$. Hence, for $x > 1 + |\eta|$,
    {\allowdisplaybreaks
    \begin{align*}
        &f(x + \eta) - f(x)
        = \eta\{p x^{p-1}(\log x)^q + qx^{p-1}(\log x)^{q-1}\} + \frac{1}{2}\eta^2\{p(p-1)x^{p-2}(\log x)^q \\
        &+ pqx^{p-2}(\log x)^{q-1} + q(p-1)x^{p-2}(\log x)^{q-1} + q(q-1)x^{p-2}(\log x)^{q-2}\}
        \\
        &+ pO(x^{p-3}(\log x)^q) + qO(x^{p-3}(\log x)^{q-1}) \\
        &= p x^{p-2}(\log x)^q \left\{x\eta + \frac{p-1}{2}\eta^2 + O\!\left(\frac{1}{x}\right)\right\}\\
        &+ qx^{p-2}(\log x)^{q-1}\left\{x\eta + \frac{2p-1}{2}\eta^2 + \frac{q-1}{2\log x}\eta^2 + O\!\left(\frac{1}{x}\right)\right\}.
    \end{align*}}

    Substituting this into \eqref{eq:gen for reactions}, we obtain, for large $x \in \S$,
    \begin{align*}
        \A f(x)
        &= p x^{p-2}(\log x)^q \left\{m(x)x + (p-1)v(x) + O\!\left(\frac{v(x)}{x}\right)\right\} \\
        &\quad + q x^{p-2}(\log x)^{q-1} \left\{m(x)x + (2p-1)v(x) + (q-1)\frac{v(x)}{\log x} + O\!\left(\frac{v(x)}{x}\right)\right\}.
    \end{align*}

\vspace{0.2cm}

  \noindent 2. For the exponential function $f(x) = e^{px}$, we have
    \begin{align*}
        \A f(x)
        &= \sum_{\eta \in \Gamma} (e^{p(x+\eta)} - e^{px})\lambda_\eta(x)
        = e^{px}\sum_{\eta \in \Gamma}(e^{p\eta}-1)\lambda_\eta(x) \\
        &= e^{px}\sum_{\eta \in \Gamma}\left\{p\eta + \frac{1}{2}(p\eta)^2 + O((p\eta)^3)\right\}\lambda_\eta(x) \\
        &= e^{px}\{pm(x) + p^2 v(x) + O(p^3)O(v(x))\}, \quad \text{as } p \to 0, \ x \to \infty.
    \end{align*}


\noindent \textit{Proof of Corollary \ref{cor: for the sup assumption}.}
   Suppose first that $\dlimsup_{x\to \infty}H_1(x)<1$.
Then one can choose $\epsilon>0$ small enough so that $\dlimsup_{x \to \infty}H_1(x) \le 1-\epsilon$.
Consider the function $f(x) = (\log x)^{\epsilon/2}\mathbbm{1}_{\{x>1\}}+\mathbbm{1}_{\{x\le 1\}}$. 
By Proposition~\ref{prop: generators}, we obtain
\begin{align}\label{eq: log x f 2}
    \A f(x) =  \frac{\epsilon}{2}\frac{v(x)}{x^2(\log x)^{2-(\epsilon/2)}}\left\{H_1(x)-(1-\epsilon)-\epsilon/2+O\left(\frac{\log x}{x}\right)\right\}. 
\end{align}
Since $v(x)>0$ for $x\in\S$, we have $\A f(x)<0$ for all sufficiently large $x$.
Taking $\dlimsup_{x\to\infty}$ on both sides of \eqref{eq: log x f 2} yields \eqref{eq: limsup1}.

If instead $\dliminf_{x\to \infty}H_1(x)>1$, choose $\epsilon>0$ small enough so that $\dliminf_{x\to \infty}H_1(x)\ge 1+\epsilon$, and define
$f(x) = (\log x)^{-\epsilon/2}\mathbbm{1}{{x>1}}+\mathbbm{1}{{x\le 1}}$.
Then Proposition~\ref{prop: generators} gives
\begin{align*}
     \A f(x) =  -\frac{\epsilon}{2}\frac{v(x)}{x^2(\log x)^{2+(\epsilon/2)}}\left\{H_1(x)-(1+\epsilon)+\epsilon/2+O\left(\frac{\log x}{x}\right)\right\}.
\end{align*} 
The desired inequality \eqref{eq: limsup2} then follows by the same reasoning.


\begin{thebibliography}{10}

\bibitem{jia2020dynamical}
Chen Jia and Ramon Grima.
\newblock Dynamical phase diagram of an auto-regulating gene in fast switching conditions.
\newblock {\em The Journal of chemical physics}, 152(17), 2020.

\bibitem{holehouse2020stochastic}
James Holehouse, Augustinas Sukys, and Ramon Grima.
\newblock Stochastic time-dependent enzyme kinetics: Closed-form solution and transient bimodality.
\newblock {\em The Journal of Chemical Physics}, 153(16), 2020.

\bibitem{bo2017multiple}
Stefano Bo and Antonio Celani.
\newblock Multiple-scale stochastic processes: decimation, averaging and beyond.
\newblock {\em Physics reports}, 670:1--59, 2017.

\bibitem{segel1975enzyme}
Irwin~H Segel.
\newblock {\em Enzyme kinetics: behavior and analysis of rapid equilibrium and steady state enzyme systems}, volume 115.
\newblock Wiley New York, 1975.

\bibitem{lo2024strong}
Chak~Hei Lo, Mikhail~V Menshikov, and Andrew~R Wade.
\newblock Strong transience for one-dimensional markov chains with asymptotically zero drifts.
\newblock {\em Stochastic Processes and their Applications}, 170:104260, 2024.

\bibitem{mao2004exponential}
Yong-Hua Mao and Yu-Hui Zhang.
\newblock Exponential ergodicity for single-birth processes.
\newblock {\em Journal of applied probability}, 41(4):1022--1032, 2004.

\bibitem{XuHansenWiuf2022}
Chuang Xu, Mads~Christian Hansen, and Carsten Wiuf.
\newblock Full classification of dynamics for one-dimensional continuous-time markov chains with polynomial transition rates.
\newblock {\em Advances in Applied Probability}, pages 1--35, 2022.

\bibitem{ballif2025partition}
Guillaume Ballif, Laurent Pfeiffer, and Jakob Ruess.
\newblock A partition method for bounding continuous-time markov chain models of general reaction network.
\newblock {\em arXiv preprint arXiv:2505.18735}, 2025.

\bibitem{song2021universally}
Yun~Min Song, Hyukpyo Hong, and Jae~Kyoung Kim.
\newblock Universally valid reduction of multiscale stochastic biochemical systems using simple non-elementary propensities.
\newblock {\em PLOS Computational Biology}, 17(10):e1008952, 2021.

\bibitem{feliu2013simplifying}
Elisenda Feliu and Carsten Wiuf.
\newblock Simplifying biochemical models with intermediate species.
\newblock {\em Journal of the royal society interface}, 10(87):20130484, 2013.

\bibitem{capellettie2016intermediate}
Daniele Capelletti and Carsten Wiuf.
\newblock Elimination of intermediate species in multiscale stochastic reaction networks.
\newblock {\em The Annals of Applied Probability}, 26(5):2915--2958, 2016.

\bibitem{chen2004markov}
Mufa Chen.
\newblock {\em From Markov chains to non-equilibrium particle systems}.
\newblock World scientific, 2004.

\bibitem{MT-LyaFosterIII}
Sean~P. Meyn and Richard~L. Tweedie.
\newblock {Stability of Markovian Processes III : Foster-Lyapunov Criteria for Continuous-Time Processes}.
\newblock {\em Advances in Applied Probability}, 25(3):518--548, 1993.

\bibitem{anderson2018some}
David~F. Anderson and Jinsu Kim.
\newblock Some network conditions for positive recurrence of stochastically modeled reaction networks.
\newblock {\em SIAM Journal on Applied Mathematics}, 78(5):2692--2713, 2018.

\bibitem{karlin1957classification}
Samuel Karlin and James McGregor.
\newblock The classification of birth and death processes.
\newblock {\em Transactions of the American Mathematical Society}, 86(2):366--400, 1957.

\bibitem{abramov2021necessary}
VYACHESLAV~M ABRAMOV.
\newblock Necessary and sufficient conditions for the convergence of positive series.
\newblock {\em Journal of Classical Analysis}, 19(2), 2021.

\bibitem{anderson2025new}
David~F Anderson, Daniele Cappelletti, Wai-Tong~Louis Fan, and Jinsu Kim.
\newblock A new path method for exponential ergodicity of markov processes on, with applications to stochastic reaction networks.
\newblock {\em SIAM Journal on Applied Dynamical Systems}, 24(2):1668--1710, 2025.

\bibitem{norris1997jr}
S~Norris.
\newblock Markov chains, 1997.

\bibitem{MT-LyaFosterI}
Sean~P. Meyn and Richard~L. Tweedie.
\newblock {Stability of Markvian Processes I:Criteria for Discrete-Time Chains}.
\newblock {\em Advances in Applied Probability}, 24:542--574, 1992.

\bibitem{MT-LyaFosterII}
Sean~P. Meyn and Richard~L. Tweedie.
\newblock {Stability of Markovian Processes II: Continuous-Time Processes and Sampled Chains}.
\newblock {\em Advances in Applied Probability}, 25(3):487, 1993.

\bibitem{menshikov2014explosion}
Mikhail Menshikov and Dimitri Petritis.
\newblock Explosion, implosion, and moments of passage times for continuous-time markov chains: a semimartingale approach.
\newblock {\em Stochastic Processes and Their Applications}, 124(7):2388--2414, 2014.

\bibitem{hairer2010convergence}
Martin Hairer.
\newblock Convergence of markov processes.
\newblock {\em Lecture notes}, 18(26):11, 2010.

\bibitem{NorrisMC97}
James Norris.
\newblock {\em {Markov Chains}}.
\newblock Cambridge University Press, 1997.

\bibitem{kingman1964}
J.F.C. Kingman.
\newblock The stochastic theory of regenerative events.
\newblock {\em Zeitschrift für Wahrscheinlichkeitstheorie und Verwandte Gebiete}, 2(3):180--224, 1964.

\bibitem{tweedie1981criteria}
Richard~L Tweedie.
\newblock Criteria for ergodicity, exponential ergodicity and strong ergodicity of markov processes.
\newblock {\em Journal of Applied Probability}, 18(1):122--130, 1981.

\bibitem{kim2025path}
Minjun Kim and Jinsu Kim.
\newblock A path method for non-exponential ergodicity of markov chains and its application for chemical reaction systems.
\newblock {\em Journal of Statistical Physics}, 192(6):74, 2025.

\bibitem{marangoni2003enzyme}
Alejandro~G Marangoni.
\newblock {\em Enzyme kinetics: a modern approach}.
\newblock John Wiley \& Sons, 2003.

\bibitem{johnson2011original}
Kenneth~A Johnson and Roger~S Goody.
\newblock The original michaelis constant: translation of the 1913 michaelis--menten paper.
\newblock {\em Biochemistry}, 50(39):8264--8269, 2011.

\bibitem{briggs1925note}
George~Edward Briggs and John Burdon~Sanderson Haldane.
\newblock A note on the kinetics of enzyme action.
\newblock {\em Biochemical journal}, 19(2):338, 1925.

\bibitem{mao2004ergodic}
Yonghua Mao.
\newblock Ergodic degrees for continuous-time markov chains.
\newblock {\em Science in China Series A: Mathematics}, 47:161--174, 2004.

\bibitem{Gillespie77}
Daniel~T. Gillespie.
\newblock {Exact stochastic simulation of coupled chemical reactions}.
\newblock {\em J. Phys. Chem.}, 81(25):2340--2361, 1977.

\bibitem{sonnad2004solution}
JR~Sonnad and Chetan~T Goudar.
\newblock Solution of the haldane equation for substrate inhibition enzyme kinetics using the decomposition method.
\newblock {\em Mathematical and computer modelling}, 40(5-6):573--582, 2004.

\bibitem{haldane1930enzymes}
JBS Haldane.
\newblock Enzymes longmans.
\newblock {\em Green and Co, UK}, 7, 1930.

\bibitem{kurtz1970solutions}
Thomas~G Kurtz.
\newblock Solutions of ordinary differential equations as limits of pure jump markov processes.
\newblock {\em Journal of applied Probability}, 7(1):49--58, 1970.

\bibitem{AndKurtz2011}
David~F. Anderson and Thomas~G Kurtz.
\newblock {Continuous time Markov chain models for chemical reaction networks}.
\newblock In H~Koeppl Et~al., editor, {\em Design and Analysis of Biomolecular Circuits: Engineering Approaches to Systems and Synthetic Biology}, pages 3--42. Springer, 2011.

\end{thebibliography}
\end{document}